\documentclass[12pt]{amsart}

\usepackage{amsfonts,amssymb,verbatim,amsmath,amsthm,latexsym,textcomp,amscd, hyperref}

\usepackage{graphicx}
\usepackage{color}
\usepackage{url}

\begin{document}

\hyphenation{hypo-thesis}

\newtheorem{theorem}{Theorem}[section]
\newtheorem{thm}{Theorem}[section]
\newtheorem{prop}[theorem]{Proposition}
\newtheorem{proposition}[theorem]{Proposition}
\newtheorem{lemma}[theorem]{Lemma}
\newtheorem{cor}[theorem]{Corollary}
\newtheorem{definition}[theorem]{Definition}
\newtheorem{conj}[theorem]{Conjecture}
\newtheorem{claim}[theorem]{Claim}
\newtheorem{defth}[theorem]{Definition-Theorem}
\newtheorem{example}[theorem]{Example}
\newtheorem{obs}[theorem]{Observation}
\newtheorem{rmk}[theorem]{Remark}

\newtheorem{introthm}{Theorem}
\newtheorem{introthma}{Theorem}
\newtheorem{introcor}[introthm]{Corollary}
\renewcommand{\theintrothm}{\Alph{introthm}}
\renewcommand{\theintrocor}{\Alph{introcor}}
\renewcommand{\theintrothma}{\Alph{introthma}$'$}
 
\newcommand{\dd}{{\partial}}
\newcommand{\s}{{\sigma}}

\newcommand\AAA{{\mathcal A}}
\newcommand\BB{{\mathcal B}}
\newcommand\CC{{\mathcal C}}
\newcommand\C{{\mathbb C}}
\newcommand{\Chat}{{\hat {\mathbb C}}}
\newcommand\DD{{\mathcal D}}
\newcommand\EE{{\mathcal E}}
\newcommand\G{{\Gamma}}
\newcommand\Gr{{\mathcal G}}
\newcommand\HH{{\mathcal H}}
\newcommand{\HHH}{{\mathbb H}}

\newcommand\Hyp{{\mathbb H}}
\newcommand\FF{{\mathcal F}}
\newcommand\GG{{\mathcal G}} 
\newcommand\MM{{\mathcal M}}

\newcommand\K {{\mathcal K}}
\newcommand\LL{{\mathcal L}}
\newcommand\XX{{\mathcal X}}
\newcommand\YY{{\mathcal Y}}

\newcommand{\N}{{\mathcal N}}
\newcommand\PP{{\mathcal P}}

\newcommand\TT{{\mathcal T}}

\newcommand{\PSL}{{PSL_2 (\mathbb{C})}}
 \newcommand\til{\widetilde}
\newcommand\CH{{\CC\HH}}

\newcommand\EXH{{ \EE (X, \HH )}}
\newcommand\GXH{{ \GG (X, \HH )}}
\newcommand\GYH{{ \GG (Y, \HH )}}
\newcommand\PEX{{\PP\EE  (X, \HH , \GG , \LL )}}
\newcommand\MF{{\MM\FF}}

\def\mul{\stackrel{{}_\ast}{\asymp}}
\def\add{\stackrel{{}_+}{\asymp}}
\def\ladd{\stackrel{{}_+}{\prec}}
\def\gadd{\stackrel{{}_+}{\succ}}
\def\lmul{\stackrel{{}_\ast}{\prec}}
\def\gmul{\stackrel{{}_\ast}{\succ}}

\def\co{{\colon \thinspace}}

 \newcommand{\Isom} {\operatorname {Isom}} 
\newcommand{\Stab} {\operatorname {Stab}} 
\newcommand{\Int} {\operatorname {Int}} 
\newcommand{\Ax} {\operatorname {Ax}}
\newcommand{\Teich} {\operatorname {Teich}} 
 \newcommand{\Mod} {\operatorname {Mod}} 
  \newcommand{\ML} {\operatorname {ML}} 
  \newcommand{\PML} {\operatorname {PML}}

\title[Limits of  limit sets II]{Limits of   limit sets II: \\ Geometrically Infinite Groups}

\author{Mahan Mj}
\address{\begin{flushleft} \rm {\texttt{mahan@rkmvu.ac.in; mahan.mj@gmail.com \\http://people.rkmvu.ac.in/$\sim$mahan/} }\\ School
of Mathematical Sciences, RKM Vivekananda University\\
P.O. Belur Math, Dt. Howrah, WB 711202, India \end{flushleft}}

\author{Caroline Series}
\address{\begin{flushleft} \rm {\texttt{C.M.Series@warwick.ac.uk \\http://www.maths.warwick.ac.uk/$\sim$masbb/} }\\ Mathematics Institute, 
 University of Warwick \\
Coventry CV4 7AL, UK \end{flushleft}}


\thanks{Research of first author partially supported by  CEFIPRA project 4301-1}   

\date{\today}

 \begin{abstract}
We show that for a strongly convergent sequence of purely loxodromic finitely generated Kleinian groups with incompressible ends,  Cannon-Thurston maps, viewed as maps from a fixed base limit set to the Riemann sphere, converge uniformly.
For algebraically convergent sequences we show that there exist examples where even pointwise 
convergence of
Cannon-Thurston maps fails.

\medskip

\noindent {\bf MSC classification: 30F40; 57M50 }    
 
\end{abstract}

 \maketitle
 

 \tableofcontents

 \section{Introduction}
Given an isomorphism between Kleinian groups,  a Cannon-Thurston map is a  continuous equivariant  map between their limit sets. It is by no means obvious that such a map always exists, however as the culmination of a long series of developments, it was shown in  \cite{mahan-kl}
that given a weakly type preserving (see below) isomorphism between any geometrically finite group $\Gamma$ and any Kleinian group $G$,   a Cannon-Thurston map always exists. 

 This paper is the second of two dealing with convergence of Cannon-Thurston or $CT$-maps, considered as
a sequence of continuous maps  from   the limit set of a 
 fixed geometrically finite group  to the sphere.   The  main questions addressed in both papers are:
\begin{enumerate}
\item Does strong convergence of finitely generated Kleinian groups imply uniform convergence of $CT$-maps?
\item Does algebraic convergence of finitely generated Kleinian groups imply pointwise convergence of $CT$-maps?
\end{enumerate}
In the first paper \cite{mahan-series1}
we dealt with the geometrically finite case  by showing that  both questions have a positive answer for a sequence of  geometrically finite groups converging to a geometrically finite limit, provided that, in case (2), the geometric limit is also geometrically finite.  As observed in  \cite{mahan-series1}, it is easy to see that if the groups converge algebraically but not strongly, then uniform convergence necessarily fails.

In the present paper we study the situation in which the limit group is geometrically infinite.
We show that, in the absence of parabolics and with incompressible ends, the answer to (1) is always positive, but, 
in what is   the most unexpected outcome of  our investigations, we provide a counter example to (2) by exhibiting a sequence of geometrically finite groups  converging algebraically but not strongly, for which the corresponding $CT$-maps fail to converge pointwise at a countable set of  points. The class of limit groups in question are Brock's partially degenerate examples~\cite{brock-itn}, described in more detail below. Thus our second main result  answers in the negative 
the second part of Thurston's  Problem 14 in his seminal paper~\cite{thurston}.     In these examples, both the algebraic and geometric limits of the $G_n$ are geometrically infinite. 
We do not know whether there exist examples of non-convergence in which the algebraic limit is geometrically finite but the geometric limit is not.

 Recall that an isomorphism $\rho: \G \to G$ between Kleinian groups  is  \emph{strictly type preserving} if $\rho(\gamma) \in G$  is parabolic if and only if $\gamma \in \G$ is also  parabolic; it is \emph{weakly type preserving} if the image of any parabolic element is parabolic. Since $CT$-maps preserve fixed points,  is easy to see  that a necessary criterion for the existence of a $CT$-map $\hat i \co  \Lambda_{\G} \to \Lambda_{G}$   between limit sets is that $\rho$ be weakly type preserving.

Our first main result, largely answering question (1), is:
\begin{introthm}\label{thm:strong=unif}  
Let $\Gamma$ be a  geometrically finite   Kleinian group  without parabolics,   which does not split as a free product. Let $\rho_n \co \Gamma  \to G_n$ be a sequence of strictly type preserving
isomorphisms to geometrically finite Kleinian groups $G_n$, which converge strongly to a  totally degenerate purely loxodromic Kleinian group $G_{\infty} = \rho_{\infty} (\Gamma )$. Then the sequence of $CT$-maps $\hat i_n: \Lambda_{\Gamma} \to  \Lambda_{G_n}$ converges uniformly to $\hat i_{\infty}: \Lambda_{\Gamma} \to  \Lambda_{G_{\infty}}$.
\end{introthm}
This result was proved by Miyachi~\cite{miyachi} in the case in which $\G$ is a surface group without parabolics and the injectivity radius is uniformly bounded below along the whole sequence. The condition that   $\G$ does not split as a free product is  equivalent to requiring that all ends of the  manifold $\HHH^3/\G$ are incompressible, see~\cite{bon}.

Theorem A  of \cite{mahan-series1}, which is essentially the above result in the geometrically finite case,  does not have  any of the restrictions  (absence of parabolics, not splitting as a free product,  strictly type preserving, totally degenerate) imposed above.    We  introduce these restrictions largely because of  technical  issues concerning  the model for the ends of the limit manifold $\HHH^3/G_{\infty}$.   With a bit more work, similar techniques to those used here can be used to prove the theorem in the general case, see~\cite{mahan-series3}.

If the convergence is algebraic but not strong, then uniform convergence necessarily fails.
This is an immediate consequence of Evans' theorem~\cite{evans1, evans2} that the limit sets $\Lambda_{G_n}$ converge  in the Hausdorff metric to the limit set of the \emph{geometric} limit   in the Hausdorff metric, see~\cite{mahan-series1} for further discussion.
 As far as we know, the question of pointwise convergence in this situation has not  previously been addressed. In answer to question (2) we have:
 \begin{introthm}\label{thm:alg=ptwise}  
 Let $\Gamma  $ be a  Fuchsian group for which  $\HHH^2/\G$ is a closed surface of genus at least $2$. Then there exists a Kleinian group $G_{\infty} $, together with an  isomorphism  $\rho_{\infty} : \Gamma \to G_{\infty} = \rho_{\infty} (\Gamma )$,   and a sequence of representations  $\rho_n\co \Gamma \to G_n$  to geometrically finite groups converging algebraically  to $G_{\infty} $, such that  the sequence of  {CT}-maps $\hat i_n: \Lambda_{\Gamma} \to  \Lambda_{G_n}$ fails to converge  pointwise to  $\hat i_\infty$ at a countable set of points in $\Lambda_{\G}$.
\end{introthm}

Implicit in the statement of  Theorem~\ref{thm:strong=unif} is the existence of the $CT$-map from $\Lambda_{\G}$ to $\Lambda_{G_{\infty}}$. This result has a long history which we do not intend to repeat in detail here.
 The most general result, in which $G = G_{\infty}$ is an arbitrary torsion free non-elementary Kleinian group, can be found in~\cite{mahan-kl}. The restricted case in which $\G$ is a surface group and $G$ is singly or doubly degenerate, is the main result of~\cite{mahan-split}, see Section~\ref{sec:thmA} below.
The original seminal case in which $M = \HHH^3/G$ is the cyclic cover of a  $3$-manifold fibering over  the circle with fibers a closed surface is of course due to Cannon and Thurston~\cite{CT}; this was extended to the case in which $G$ is a surface group with  a lower bound on the lengths of loxodromics in~\cite{bowditch-ct} and ~\cite{mahan=jdg}, or more generally when $M$ is an arbitrary hyperbolic manifold with incompressible boundary in~\cite{mahan-pared}.    The older history in the case in which $G$ is geometrically finite is discussed in~\cite{mahan-series1}.

In~\cite{mahan-series1} we introduced general criteria for uniform and pointwise convergence of $CT$-maps, called UEPP and EPP respectively. These  compare the geometry of  the obvious embedding  of the Cayley graph of the base group $\G$ into $\HHH^3$,  to the corresponding embeddings for the groups $G_n, G_{\infty}$. The main work in this paper consists in verifying that these criteria hold (in the case of Theorem~\ref{thm:strong=unif}) or understanding why they do not (in the case of Theorem~\ref{thm:alg=ptwise}). After slightly reformulating the condition UEPP, we see that for the case of strictly type preserving maps of surface groups, the needed condition has essentially already been proved  in~\cite{mahan-split}. In order to explain this, we give in Section~\ref{sec:unbounded}  a brief outline of the relevant parts of the arguments in~\cite{mahan-split}.  For the benefit of readers who have not gone through all of this previous work, which in turn depends heavily on  the Minsky model of degenerate Kleinian groups, we preface this by briefly sketching in Section~\ref{sec:bndgeom} how the argument goes in the case of groups of bounded geometry, thus reproving Miyachi's theorem~\cite{miyachi}. Our proof in this case follows easily using the method explained in~\cite{mahan=jdg} and~\cite{mahan=ramanujan} and is independent of~\cite{mahan-split}.   

Let $R$ be a surface with boundary and $\Lambda$ a lamination on $R$. A complementary region in $R \setminus \Lambda$ is called a  \emph{crown domain}
if it contains a component  of $\partial R$. 
The counter example in Theorem~\ref{thm:alg=ptwise} arises from Brock's examples of 
a sequence of  quasi-Fuchsian groups $G_{n}$  converging algebraically but not strongly  to a partially degenerate group  $G_{\infty}$.
More precisely, we prove the following, which immediately implies Theorem~\ref{thm:alg=ptwise}:
\begin{introthm}\label{brockexample}
Fix a closed hyperbolizable surface $S$ together with  a separating simple closed curve $\sigma$, dividing $S$  into two pieces $L$ and $R$. 
Let $\alpha$ denote an automorphism of $S$ such that $\alpha |_{L}$ is the identity and $\alpha |_{R} = \chi$ is a pseudo-Anosov diffeomorphism of
$R$ fixing the boundary $\sigma$. Let $X$ be a hyperbolic structure on $S$ and let $G_n$ be the quasi-Fuchsian group given by the simultaneous uniformization of $(\alpha^n(X), X)$. Let $G_\infty$
denote the algebraic limit of the sequence $G_n$, suitably normalized by a basepoint in the lift of the lower boundary $X$. Let  $\hat i_n
 \co \Lambda_{G_0} \to \Lambda_{G_n}$, $n \in \mathbb N \cup \infty$,   be the corresponding $CT$-maps and  let $\xi \in \Lambda_{G_0}$. Then  $\hat i_n (\xi )$  converges  to $\hat i_\infty (\xi )$  if and only if $\xi$ is not   the endpoint of the lift to $\HHH^2$ of a boundary leaf, other than $\sigma$,  of
the crown domain of the unstable lamination of $\chi$, viewed as a lamination on the surface $R$.
\end{introthm}

The outline of the paper is as follows. 
In Section~\ref{sec:background} we set up background and notation, in particular reviewing briefly
 what we need
from the  theory of  hyperbolic spaces and electric geometry in~\ref{sec:relhyp}.  These
techniques are central in~\cite{mahan-split}, and are also used
here in the discussion of Theorem~\ref{brockexample}.

 In Section~\ref{sec:CTMaps} we  recall results from~\cite{mahan-series1}
on Cannon-Thurston maps, in particular we explain our convergence criterion UEPP.  
 In Section~\ref{sec:thmA} we prove Theorem~\ref{thm:strong=unif}. As discussed above, we first  give a brief discussion of a proof in the case of  bounded geometry, that is, when the injectivity radius of all manifolds in the sequence is uniformly bounded below. This is essentially Miyachi's theorem referred to above.
 We then turn to the general situation, outlining as we go the relevant steps in the proof for a single $CT$-map as in~\cite{mahan-split}.   Finally in Section~\ref{sec:brock} we explain the counter examples to pointwise convergence, explaining the Brock examples and then proving Theorem~\ref{brockexample}. 


\noindent {\bf Acknowledgments:} This work was done in part while the first author was visiting Universit{\'e} Paris-Sud XI under the Indo-French 
collaborative programme ARCUS. He gratefully acknowledges their support
and hospitality.


 \section{Background}\label{sec:background}

  \subsection{Kleinian groups}
 
 \label{sec:basics}

 A \emph{Kleinian group} $G$ is a discrete subgroup of $PSL_2 (\mathbb{C})$. As such it acts as a properly discontinuous group of isometries of hyperbolic $3$-space $\HHH^3$, whose boundary we identify with  the Riemann sphere $\Chat = \C \cup \infty$.
 As in~\cite{mahan-series1}, all groups in this paper will be finitely generated and torsion free, so that  
 $ M = \HHH^3/G$ is a hyperbolic $3$-manifold.  
  The \emph{limit set} $\Lambda_G \subset \Chat$  is the    set of accumulation points of any $G$-orbit. 
  
  A Kleinian group  is 
\emph{geometrically finite} if it has a fundamental polyhedron in $\HHH^3$ with finitely many faces; a group which is not geometrically finite is also called \emph{degenerate}. 
 The  point of this paper is to investigate the extension of~\cite{mahan-series1} to the degenerate case. We say a group is  \emph{totally degenerate} if it is not geometrically finite and $\Lambda_G= \Chat$.
The structure of degenerate groups has recently been elucidated by the work of Minsky et al.~\cite{minsky-elc1, minsky-elc2}   on the ending lamination theorem and the tameness theorem of Agol~\cite{agol} and Calegari and Gabai~\cite{calegari}. This paper rests heavily on these results.

A Kleinian group  $G$ is  a  \emph{surface group}, if there is a hyperbolic surface $S$, together with a discrete faithful representation $\rho \co \pi_1(S) \to G$. 
The corresponding manifold $  \HHH^3/G$  is homeomorphic to $S \times \mathbb R$, see \cite{bon}.
It is singly or doubly degenerate according as one or both of its ends are geometrically infinite with filling ending laminations.

\subsection{Balls and geodesics}  \label{sec:balls} We will be working in  hyperbolic space  $\HHH^n$ for $n = 2,3$. We denote the 
 hyperbolic metric on $\HHH^n$ by $d_{\HHH}$  or occasionally $d_{\HHH^n}$; sometimes we explicitly  use the ball model $\mathbb B$ with centre $O$ and denote by $d_{\mathbb E}$   the Euclidean metric  on $\mathbb B \cup \Chat$.
For $P \in \HHH^n$,  write $B (P; R)$, or when needed $B_{\HHH}(P; R)$ or even $B_{\HHH^n}(P; R)$, for the hyperbolic ball centre $P$ and  radius $R$.  Let $\beta $ be a path in $\HHH^n$ with endpoints $X,Y$. We write $[\beta]$  or $[X,Y]$ for the $\HHH^n$-geodesic  from $X$ to $Y $.

\subsection{The Cayley graph}\label{sec:cayley}

Let $G$ be a finitely generated Kleinian group with generating set $G^* = \{e_1, \ldots, e_k \}$.  We assume throughout that $G^*$ is symmetric, in the sense that $g \in G^*$ if and only if $g^{-1} \in G^*$ for any $g \in G$. The Cayley graph $\mathcal GG$  of $G$ is the graph whose vertices are elements $g \in G$ and which has an edge between $g,g'$ whenever $g^{-1}g' \in G^*$.  The graph metric $d_G$ is defined as the edge length of the shortest path between vertices so that $d_G(1,e_i) =1$ for all $i$, where $1$ is the unit element of $G$.  Let
 $|g|$ 
   denote the word length of $g \in   G$ with respect to $G^*$, so that $|g| = d_G(1,g)$. For $X \in \Gr G$, we denote by $B_{G}(X; R) \subset \Gr G$  the $d_G$-ball centre $X$ and  radius $R$.

Choose a  basepoint  $  O_G \in \HHH^3$ which is not a fixed point of any element of $G$. One may if desired assume the basepoint is the centre $O$ of the ball model $\mathbb B$ as above. For simplicity, we do this throughout the paper unless indicated otherwise. 
 Then $\mathcal GG$  is immersed in $\HHH^3$ 
by the map $j_{G}$ which sends $g \in G$ to $j_{G}(g) = g \cdot O$, and which sends the edge joining $g,g'$ to the $ \HHH^3$-geodesic joining $j_{G}(g),j_{G}(g')$. In particular,  $j_G(1) = O$. Note that using the ball model of $\HHH^3$, the limit set $\Lambda_G  $ may be regarded as the completion of $j_{G}(\mathcal GG)$  in the Euclidean metric $d_{\mathbb E} $ on   $\mathbb B \cup \Chat$.

 \subsection{Algebraic and Geometric Convergence}  \label{sec:alggeolts}
Let $\G$ be a geometrically finite Kleinian group. A sequence of group isomorphisms
 $\rho_n\co \G \to PSL_2 (\mathbb{C}),   n = 1,2 \ldots  $ is  said to converge  to the representation $\rho_{\infty}\co \Gamma \to PSL_2 (\mathbb{C})$  \emph{algebraically}
if  for each $g \in \G$,  $\rho_n(g) \to \rho_{\infty} (g)$ as elements of $PSL_2 (\mathbb{C})$.  
 The representations converge \emph{geometrically}   if  $(G_n = \rho_n(G))$ converges as a sequence of closed subsets of $\PSL$
to  $G_g \subset PSL_2 (\mathbb{C})$. Then  $G_g$ is a   Kleinian group called the
\emph{geometric limit} of $(G_n)$.    The sequence   $(\rho_n)$ converges  \emph{ strongly }
to  $\rho_{\infty}(G)$ if $\rho_{\infty}(G)= G_g$ and the convergence is both geometric and algebraic.  If a sequence of  groups converge algebraically, they have a geometrically convergent subsequence, see~\cite{marden-book} Theorem 4.4.3.

Alternatively, following Thurston \cite{thurstonnotes}, see also for example~\cite{CEG} Chapter 3, we say that a sequence of manifolds
with base-frames $(M_n, \omega_n )$ converges {\it geometrically }
(or in the $C^\infty$-Gromov-Hausdorff topology)
to a manifold with base-frame $(M_{\infty}, \omega_{\infty})$ if for each compact submanifold $C\subset M_{\infty}$ containing the base-frame $\omega_{\infty}$, there are smooth embeddings
$\psi_n: C \rightarrow M_n$ (for all sufficiently large $n$) which map base-frame  to base-frame  and such that $\psi_n$
converges to an isometry in the $C^\infty$-topology. 
 Kleinian groups $G_n$ are said to converge \emph{geometrically} to $G_{\infty}$ if the corresponding framed manifolds 
$(M_n = \HHH^3/G_n, \omega_n)$ converge geometrically
to $(M_{\infty} = \HHH^3/G_{\infty}, \omega_{\infty})$, where  the base-frames $\omega_n, \omega_{\infty}$ are all  the projection of a fixed  base-frame  in $\HHH^3$.
The sequence $((M_n, \omega_n ))$ converges  \emph{strongly}
to  $(M_{\infty}, \omega)$  if the convergence is geometric and in addition the convergence of  $(\rho_n)$  to $\rho_{\infty} $ is algebraic.  
We remark that changing the  basepoints in the above discussion may result in a different geometric limit.

The relation between these definitions is the following. Fix once and for all a standard base frame $\Omega$ in $\HHH^3$, with  basepoint at the origin $O$ in the ball model of hyperbolic $3$-space.
Given a framed manifold $(M, \omega )$, there is a unique developing map $(\widetilde M, \tilde \omega) \to \HHH^3$ (where $\widetilde M$ is the universal cover of $M$) which sends a fixed lift  $\tilde \omega$ of $\omega$ to $\Omega \in \HHH^3$. The induced  holonomy homomorphism sends $ \pi_1(M,o)$ to a discrete  torsion free subgroup  of $SL(2,\C)$, where $o \in M$ is the  basepoint of $\omega$. By for example~\cite{CEG} Theorem 3.2.9, this map is a homeomorphism  with appropriate topologies, so  that convergence of manifolds  in the sense of Thurston is equivalent to geometric convergence in the first sense defined above, see for example \cite{marden-book} Chapter 4 or \cite{kap} Theorem 8.11. 
 In particular, the map  $\psi_n: C \rightarrow M_n$  is  the projection to the quotient manifolds of  a bi-Lipschitz embedding $\tilde \psi_n \co B(O;R) \to \HHH^3$ where $B(O;R) \subset \HHH^3$ is a large ball whose projection to $\HHH^3/G_{\infty}$ contains $C$, see for example~\cite{and-can1} Lemma 9.6.

 \subsection{Scott cores}  \label{sec:cores}

 Recall that a Scott core of a $3$-manifold $V$ is a compact connected $3$-submanifold  $K_V$ such that the inclusion $K_V \hookrightarrow V$ induces an isomorphism on fundamental groups.   The Scott core is unique up to isotopy \cite{mms}. Note that in general, the Scott core may be much smaller than the convex core, even when the group is convex cocompact. We shall need the following relationship between the Scott core and the ends of $V$. 
  
 \begin{lemma}[\cite{bon} Proposition 1.3, \cite{kap} Theorem 4.126] \label{cutoffend}
Let $K_V$ be  a Scott core of a $3$-manifold $V$. There is a  bijective correspondence between the ends of $V$ and boundary components of $V \setminus K_V$.  Hence each component of $\dd K_V$ bounds a non-compact  component of $V \setminus K_V$ and each of these components is an end of $V$.
\end{lemma}

   Let $\rho_n \co \G \to \PSL$ be a sequence of representations converging strongly
   to $\rho_{\infty}$.  Fixing the base-frames as the projections $\omega_n$ of the frame $\Omega$ in $ \HHH^3$ to $M_n = \HHH^3/G_{n}$, $n \in \mathbb N \cup  \infty$, we obtain a 
   corresponding  sequence of  framed hyperbolic manifolds $(M_n , \omega_n)$  converging geometrically to $M_{\infty} = \HHH^3/G_{\infty}$.  Set $N = \HHH^3/\G$ and let $K_N$ be a Scott core of $N$, chosen such that the baseframe $\omega_N$ for $N$ has basepoint   $o_N \in K_N$.  The representations $\rho_{n}$ induce homotopy equivalences   $\phi_n \co  K_N \rightarrow M_{n}$. We can lift $\phi_n$ to maps $\tilde \phi_n \co  \til K_N \rightarrow \til M_{n}$ for $n  \in \mathbb N \cup  \infty$ and note that by our choices that  $\tilde \phi_n(\tilde \omega_N)$ converges to $\tilde \phi_{\infty}(\tilde \omega_N)$.

In general the homotopy equivalences $\phi_n$ may not be homeomorphisms.  However in the situation of strong convergence, 
the proof of  \cite{CM} Proposition 3.3 or   \cite{and-can1}  Lemma 9.7, (see also Lemma 3.6 and the first part of the proof of Theorem A in  \cite{and-can-mc}) gives: 
\begin{lemma}  \label{cores} Let $\Gamma$ be a geometrically finite group and let $\rho_n $ be a  sequence of discrete faithful representations of $\Gamma$ converging strongly to $\rho_{\infty}$. Let $ \mathcal K = K_{M_{\infty}}$ be a compact core for $M_{\infty} = \HHH^3/\rho_{\infty}(\Gamma)$.   Let $\psi_n \co  \K \to M_n$ be the  bi-Lipschitz embeddings coming from the geometric convergence, inducing maps $(\psi_n)_* \co \rho_\infty(\Gamma) \to \rho_n(\G)$. Then  for all large enough $n$,  $(\psi_n)_* = \rho_n \circ \rho^{-1}_{\infty}$ and $\psi_n ( \K) $ is a compact core for $M_n = \HHH^3/\rho_{n}(\Gamma)$.     \end{lemma}

  We remark that the hypotheses in  \cite{CM} and  \cite{and-can1} that all groups be purely hyperbolic, or indeed that the convergence be strictly type preserving,  are not needed for this lemma.
Also note that as remarked above,  in general  $\psi_n ( \K)$ may be much smaller than the convex core of $M_n$.
 If the convergence is not strong,  Lemma~\ref{cores}  may fail even when the limit group is geometrically finite, as is shown by the well known Anderson-Canary examples~\cite{and-can}.
 
Lemma~\ref{cores} shows that, in the situation of strong convergence, we may take the homotopy equivalences $\phi_n$ to  be homeomorphisms between Scott cores of the relevant groups. 
 It also allows us to identify the ends of $M_{\infty}$ with the ends of the approximating groups. We have: 
 \begin{cor}  \label{cores1}  Let $\Gamma$ be a geometrically finite group and let $\rho_n $ be a  sequence of discrete faithful representations of $\Gamma$ converging strongly to $\rho_{\infty}$. Then, up to replacing $\G$ by the group $G_{n_0}$ for some $n_0 \in \mathbb N$, we can pick a Scott core $K$ of $\HHH^3/\G$ such that 
    there are  bi-Lipschitz
embeddings $\phi_{n} \co K  \to \HHH^3/G_n$ which induce $\rho_n$, and   such that $\phi_n(K )$ is a Scott core of 
$ M_n$ for $n  \in \mathbb N \cup\infty$.  Moreover 
suppose that  $E$ is an end of $M_{\infty}$ and  $U$ is the component of $M_{\infty} \setminus  \phi_\infty(K)$ which is a neighborhood of $E$. 
 Let $F = \dd U$ and let $U_n  $ be the component of $M_n \setminus   \phi_n (K )$ bounded by  $\phi_n \phi_\infty^{-1} (F)$. Then  $U_n$ is a neighborhood of an end $E_n$ of $M_n$ and we say that $U_n$ \emph{corresponds} to $E$.   
  \end{cor}
  \begin{proof} 
  
Set $\K =   \phi_\infty(K)$. Take $n_0$ large enough that the conclusion of Lemma~\ref{cores} applies.
Replacing  $N = \HHH^3/\G$ by $M_{n_0}$, $\G$ by $G_{n_0}$, and $\rho_n$ by $\rho'_n= \rho_{n} \rho_{n_0}^{-1}$, we have a sequence of representations as before.
The core of $M_{n_0}$ can be taken to be $\psi_{n_0} ( \K) $.
 Noting  that $\rho'_n$ is induced by $\psi_{n}\psi_{n_0}^{-1} \co \psi_{n_0} ( \K) \to M_{n}$, we can replace the homotopy equivalences $\phi_n \co K_N \to M_n$ by the  homeomorphisms $\psi_{n}\psi_{n_0}^{-1} \co \psi_{n_0} ( \K) \to \psi_{n} ( \K)$ between the cores of $M_{n_0}$ and $ M_n$, $n > n_0$. These converge to the homeomorphism $\psi_{n_0}^{-1} \co\psi_{n_0} ( \K) \to \K$. In other words, we may as well  assume that the homotopy equivalences  $\phi_n$ are actually embeddings of the core $K_N$ of $N$ into $M_n, M_{\infty}$.

The idea of the final statement follows~\cite{CM} \S 8.  That  $U_n$ is a neighborhood of an end $E_n$ of $M_n$ follows from Lemma~\ref{cutoffend}. 
  \end{proof}

  \subsection{Relative Hyperbolicity and Electric Geometry}\label{sec:relhyp}

We summarize the facts we need on relative hyperbolicity  and electric geometry.
For further details, we refer the reader to
\cite{farb-relhyp, bowditch-relhyp}
see also \cite{mahan-ibdd} Section 3. 
 
Let $(X, d)$ be a $\delta$-hyperbolic metric
space, and let $\mathcal{H}$ be a collection of pairwise disjoint   subsets. 
To \emph{electrocute} $\HH$ means to construct an auxiliary metric space
$(X_{el}, d_{el})$ in which the sets in $\HH$ effectively have zero diameter, although for technical reasons it is preferable they have diameter $1$ (or $2$). Precisely, 
let $X_{el}  = X \bigcup_{H \in \mathcal{H}} (H \times [0,1])$
with $H \times \{ 0 \}$ identified to $H \subset X$. We define the electric (pseudo)-metric $d_{el}$ on $X_{el}$ as follows.
First equip $H \times  [0,1]$ with the product metric and then modify this to a pseudo-metric by quotienting so that
$H \times \{ 1 \}$ is equipped with
 the zero metric. The metric on $H \times  [0,1]$ is the path metric induced by horizontal and vertical paths.  This means that in the space $(X_{el}, d_{el})$, any two points in $H$ are at distance at most $2$. 

 \begin{definition} \cite{farb-relhyp, bowditch-relhyp}
Let $X$ be a  metric space and $\mathcal{H}$ be a collection of
mutually disjoint subsets. If $X_{el}$ is also a hyperbolic metric space, then  
 $X$ is said to be
\emph{weakly hyperbolic} relative to the collection $\mathcal{H}$.
\end{definition}

 The collection $\mathcal{H}$  is said to be
\emph{uniformly separated} if there exists $C>0$ such that  $d(H_i, H_j) \geq C$ for all
$H_i \neq H_j \in \mathcal{H}$.  
It   is \emph{uniformly quasi-convex} if there exists $C>0$ such that  if for any $H \in \HH$ and for any points $x,x' \in H$, any geodesic joining them lies within the $C$-neighborhood of $H$. 
It 
is   \emph{mutually  cobounded} if there exists $C>0$ such that 
 for all $H_i \neq H_j  \in \mathcal{H}$, $\pi_i
(H_j)$ has diameter less than $C$, where $\pi_i$ denotes a nearest
point projection of $X$ onto $H_i$.

\begin{lemma}[\cite{bowditch-relhyp}, \cite{farb-relhyp} Proposition 4.6] 
\label{lemma:weakhyp}
Let $X$ be a hyperbolic metric space and
$\mathcal{H}$ a collection of  
 uniformly quasi-convex mutually cobounded
uniformly separated subsets.  
Then $X$ is weakly hyperbolic relative to the collection $\mathcal{H}$. 
\end{lemma}

A typical example is when $X$ is hyperbolic space $\HHH^3$ and $\HH$ is the collection of lifts to $\HHH^3$ of the thin parts of a hyperbolic $3$-manifold.

Following Farb, we need to understand some finer details of the relationship between geodesics in $X$ and in $X_{el}$.  
Recall that \emph{K-quasi-geodesic} in a metric space $Y$ is a $K$-quasi-isometric embedding of an interval into $Y$, that is, a map $f:[a,b] \to Y$ such that 
$$\dfrac{1}{K} |t_1-t_2| -K \leq d_Y(f(t_1), f(t_2)) \leq K |t_1-t_2| +K$$
for all $t_1,  t_2 \in [a,b]$.  A \emph {quasi-geodesic} is a path in $Y$ which is a $K$-quasi-geodesic  for some $K>0$. 
If $X, \HH$ gives rise to an electric space $(X_{el}, d_{el})$,  then an  \emph {electric (quasi)-geodesic} in $X$ is a path in $X$ which is a
(quasi)-geodesic for the electric metric $d_{el}$.

We say that a path \emph{does not backtrack} if it does not re-enter any $H\in\mathcal{H}$ after leaving it.
Suppose that  $\lambda  $ is
an electric  quasi-geodesic  in $(X_{el},d_{el})$ without backtracking
and  with endpoints $a, b \in X \setminus \HH$. Keeping the endpoints $a,b$ fixed, replace each maximal subsegment of $\lambda $ lying within some $H \in \mathcal{H}$
by a hyperbolic  $X$-geodesic with the same endpoints.  The resulting
connected  path  is called an {\em electro-ambient quasi-geodesic} in
$X$, see Figure~\ref{fig:electro-ambient}. 
The main result we need is:

\begin{figure}[hbt] 
\centering 
\includegraphics[height=3cm]{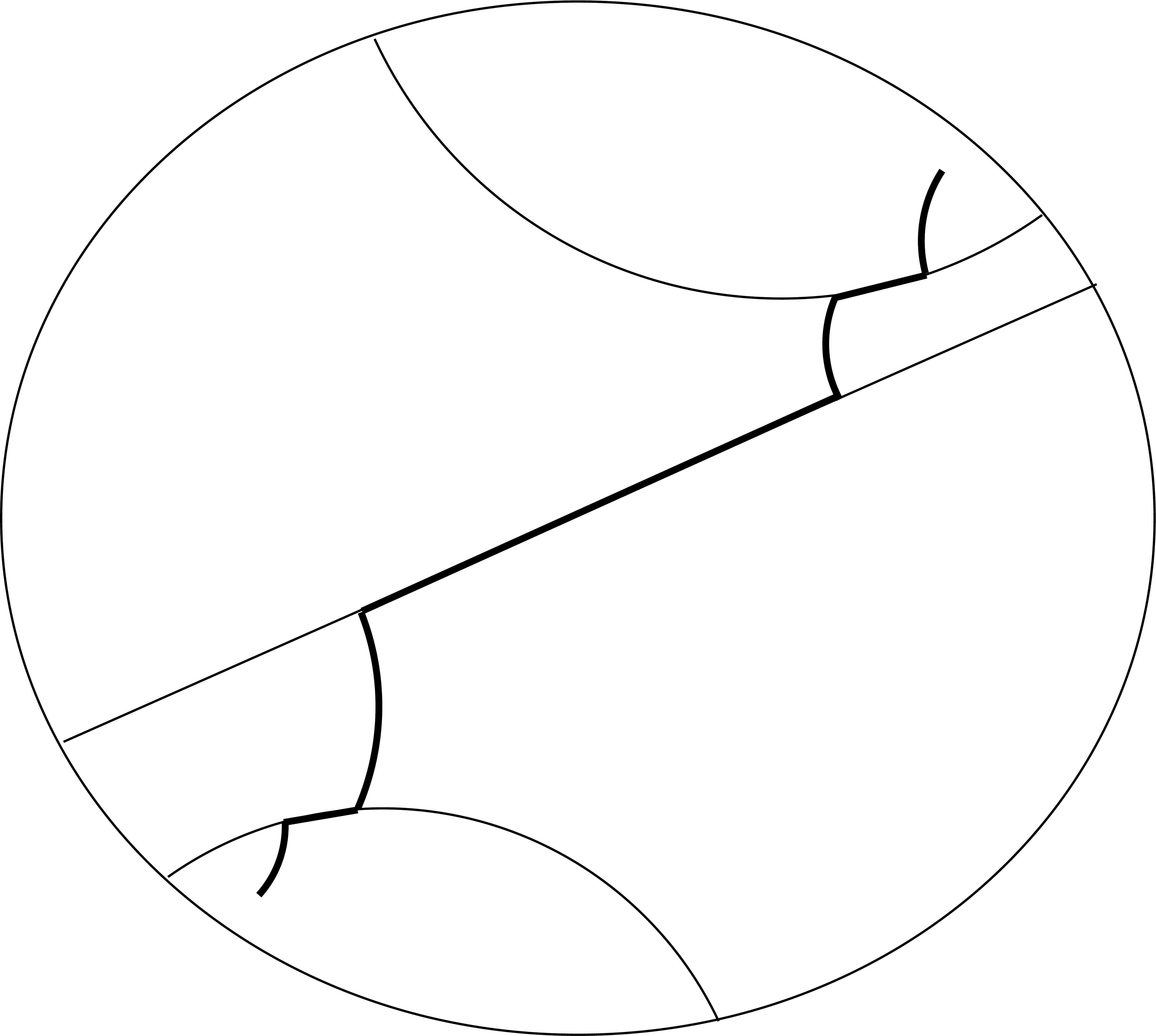}
\caption{An electro-ambient quasi-geodesic. }
\label{fig:electro-ambient}
\end{figure} 

 \begin{lemma} [\cite{mahan-ibdd} Lemma 3.7]\label{ea-genl}
Let  $X$ be  a hyperbolic metric space and  let $\HH$ be a collection of mutually cobounded uniformly separated uniformly quasi-convex sets. 
Let $\gamma$ be an electro-ambient quasi-geodesic with endpoints $a,b \in X \setminus \HH$. Then $\gamma$   is a quasi-geodesic in $X$ and   lies within bounded distance of any $X$-geodesic with the same endpoints. \end{lemma}

 In the case  in which $X$ is hyperbolic space $\HHH^3$ and $\HH$ is the collection of lifts to $\HHH^3$ of the thin parts of a hyperbolic $3$-manifold,  the proof of Lemma~\ref{ea-genl} is
straightforward and was done from first principles in  \cite{mahan-series1} Lemmas 7.15, 7.16.

\begin{rmk}\label{stronglyrelhyp} {\rm One can introduce a further property
 of a metric space $X$  being \emph{strongly hyperbolic} relative to a collection $\HH$, see \cite{mahan-ibdd} Section 3 
 and also \cite{bowditch-relhyp}. This condition concerns how paths penetrate the sets in $\HH$. If $X$ is itself a $\delta$-hyperbolic space, then the  conditions   that the sets in $\HH$ be mutually cobounded, uniformly separated and uniformly quasi-convex  are enough to imply that $X$ is strongly hyperbolic  relative to  $\HH$, see  \cite{farb-relhyp} \S 4 and  \cite{mahan-ibdd} \S 3. Since all that we needed here is the result of Lemma~\ref{ea-genl}, we do not digress to give the precise definition here.} 
\end{rmk}

\section{Cannon-Thurston Maps and Convergence Criteria}\label{sec:CTMaps}

Let $\G$ be a Kleinian group and let $\rho: \G \to \PSL$ with $\rho(\G) = G$.
Let $\Lambda_{\G}, \Lambda_G$ be the corresponding limit sets.
A \emph{Cannon-Thurston} or $CT$-map   is an equivariant continuous map  $\hat i: \Lambda_{\G} \to  \Lambda_G$, that is, a map such that 
$$ \hat i    (g \cdot \xi)  = \rho(g) \hat i (\xi) \ \ {\rm for \ all} \ \ g \in \G, \xi \in \Lambda_{\G}.$$

Recall from Section~\ref{sec:cayley} the natural embedding $j_{\G} $ of the Cayley graph  of $\mathcal G \G$ into $\HHH^3$.  
The $CT$-map $\hat i = \hat i(\rho)$ can also  be defined as the continuous extension to $\Lambda_{\G} \subset \dd \HHH^3$ of the obvious map $i\co j_{\G}(\mathcal G \G) \to \HHH^3$ defined by $ i(j_{\G}(g)) = \rho(g) \cdot O$.

Suppose alternatively that $N$ is a geometrically finite manifold   homotopy equivalent to another hyperbolic manifold $M$ and let $\phi\co K_N \to M$ be a homotopy equivalence between a Scott core $K_N$ of $N$ and $M$.
It is easy to see that   $\hat i$ is an extension  to 
$\dd \HHH^3$ 
of any  lifting $\widetilde \phi\co \widetilde K_N \to \widetilde M$, since any fixed orbit of the action  of $G$ on $\tilde \phi  ( \tilde K_N )$ can serve as a substitute for the orbit of the  basepoint $O$. Notice however that the careful discussion in~\ref{sec:alggeolts} is needed to fix basepoints if we are dealing with a sequences of groups.

In~\cite{mahan-series1} we reproved Floyd's results~\cite{floyd} on the existence of  $CT$-maps for geometrically finite groups:
\begin{theorem}[\cite{mahan-series1} Theorem 4.2]  \label{firstresult}Let $\G, G$ be finitely generated geometrically finite groups and let $\phi: \G \to G$ be weakly type preserving isomorphism. Then the \emph{CT}-map
$\hat i: \Lambda_{\G} \to \Lambda_G$ exists.   
\end{theorem} 
  
Note that the examples in~\cite{and-can} show that there exist   geometrically finite non-cusped manifolds which are homotopic but not homeomorphic, see also Lemma~\ref{cores}. We deduce from the above that the $CT$-map between their limit sets nonetheless exists and is a homeomorphism.

 \subsection{Criteria for convergence}

We now collect some criteria for the existence and convergence of $CT$-maps from~\cite{mahan-series1}.

\begin{theorem} [\cite{mahan-series1} Theorem 4.1] \label{crit1} Let $\rho \co \G \to G$ be a weakly type preserving  isomorphism of finitely generated Kleinian groups and suppose that $\G$ is geometrically finite. The $CT$-map  $  \Lambda_{\G} \to \Lambda_G$
 exists if and only if    there exists a function  $f \co \mathbb N \to \mathbb N$, such that 
 $f(N)\rightarrow\infty$ as $N\rightarrow\infty$, and such that whenever  $\lambda$ is a $d_{\G}$-geodesic segment  lying outside $B_{\G}(1; N)$ in $\Gr \G$,  the  $\HHH^3$-geodesic joining
the endpoints of $i(j_{\G}(\lambda))$ lies outside $B_{\HHH}(O; f(N))$ in $\HHH^3$.
\end{theorem}

In~\cite{mahan-split}, the criterion was used in an alternative form which involves geodesics in $\HHH^3$ for the domain group $\G$ also. Recall that a geometrically finite group is convex cocompact if its convex core is compact. In this case we can take the Scott core to be the convex core.

 \begin{theorem}\label{crit1'} Let $\rho \co \G \to G$ be a strictly type preserving  isomorphism of finitely generated Kleinian groups, and suppose that $\G$ is convex cocompact. 
Suppose that $K_N$ is the convex core of $N = \HHH^3/\G$ and that $\phi\co  K_N \to \HHH^3/G$ is a homotopy equivalence.
Then the $CT$-map  $  \Lambda_{\G} \to \Lambda_G$
 exists if and only if    there exists a non-negative function  $f \co \mathbb N \to \mathbb N$, such that 
 $f(M)\rightarrow\infty$ as $M\rightarrow\infty$, and such that whenever  $\lambda$ is a $\HHH^3$-geodesic segment  lying outside $B (O; M)$,  the  $\HHH^3$-geodesic   $[\tilde \phi ( \lambda)]$ lies outside $B(\tilde \phi(O); f(M))$, where $\tilde \phi$ is a fixed lift of $\phi$ to the obvious map from $\widetilde {K_N}$ (the convex hull of $\Lambda_{\G}$) to $\HHH^3$.
\end{theorem}

Note that to make a sensible statement here we need to insist that,
 in addition to being compact,  the 
core $K_N$ of $N$ should also be convex. If $\G$ is  Fuchsian group corresponding to a closed surface, the main case in~\cite{mahan-split}, then this is not an issue.

If $\G$ is convex cocompact, then to compare geodesics in $\Gr \G$ and $\HHH^3$, one can use that any $d_{\G}$-geodesic in $\Gr \G$ is a quasi-geodesic in $\HHH^3$, see for example~\cite{mahan-series1} Corollary 3.8. Moreover any  geodesic  in $\HHH^3$ can be tracked at bounded distance  by a path which is geodesic  in $\Gr \G$ (with bound depending only on $\G$), see~\cite{mahan-series1} Lemma 3.6.  This shows the equivalence of the criteria  in Theorems~\ref{crit1} and~\ref{crit1'}
 in this case. 

\begin{rmk}{\rm 
If $\G$ is not convex cocompact then the comparison between geodesics in $\Gr \G$ and $\HHH^3$  is not quite straightforward,  since a geodesic path in $\Gr \G$ can track round the boundary of a horosphere. 
This can be dealt with but requires more care. In Theorem~\ref{thm:strong=unif} of this paper, $\G$ is always convex cocompact, and in any case,    we shall mainly stick to the first version of the criterion.}
\end{rmk}

Now let $\G$ be a fixed geometrically finite Kleinian group and suppose that
 $\rho_n\co  \G \rightarrow \PSL$ is a sequence
of discrete faithful weakly type preserving representations converging  algebraically    to $\rho_\infty \co  \G \rightarrow \PSL $. Let $G_n = \rho_n (\G), n  \in \mathbb N \cup  {\infty}$ and write $\Lambda_n$ for $ \Lambda_{G_n}$. 
To normalize, we embed all the Cayley graphs with the same  base-point $O = O_{G_n}$
for all $n$ and set $j_{n}(g)= \rho_n(g) \cdot O, g \in \Gr \G$.
Define  $i_n \co j_{\G}(\Gr \G) \to \HHH^3 $    by $i_n(j_{\G}(g)) \mapsto \rho_{n}(g) \cdot O, g \in \Gr \G$, so that   the \emph{CT}-map  $\hat i_n: \Lambda_{\G} \to  \Lambda_{n}$ is the continuous extension of $i_n$ to $\Lambda_{\G}$.  Assuming they exist, we say that the \emph{CT}-maps   $\hat i_n: \Lambda_{\G} \to  \Lambda_{n}$ converge uniformly (resp. pointwise) to $\hat i_{\infty}$ if they do so as maps from 
$\Lambda_{\G}$ to $\Chat $. 
If $\lambda$ is any $d_{\G}$-geodesic segment in $\Gr \G$  with endpoints $\gamma, \gamma' \in  \G$, then we also write $[j(\lambda)]$ for  the $\HHH^3$-geodesic $[j(\gamma),j( \gamma')]$.

In~\cite{mahan-series1} we introduced two  properties  UEP (Uniform Embedding of Points)  and UEPP (Uniform Embedding of Pairs of Points) of the sequence $(\rho_n)$. The first was shown in~\cite{mahan-series1} to be equivalent to strong convergence, and the second is our criterion for uniform convergence of $CT$-maps. We summarise these definitions and results here.

\begin{definition}   Let $\rho_n\co \G \to G_n$ be  a sequence of weakly type preserving  isomorphisms of  Kleinian groups.   Then $(\rho_n)$ is said to satisfy
  UEP   if there
exists a non-negative function  $f \co \mathbb N \to \mathbb N$, with
 $f(N)\rightarrow\infty$ as $N\rightarrow\infty$, such that for all $g
\in \Gamma$, 
 $d_\Gamma (1,g) \geq N$ implies  $d_{\Hyp} (\rho_n(g)\cdot O , O) \geq
f(N)$ for all $n  \in \mathbb N$.  
\end{definition}

\begin{definition}\label{crit:unifcrit1} Let $\rho_n\co \G \to G_n$ be  a sequence of weakly type preserving  isomorphisms of  Kleinian groups.  
 Then $(\rho_n)$ satisfies UEPP  if there exists a  function  $f' \co \mathbb N \to \mathbb N$, such that 
 $f'(N)\rightarrow\infty$ as $N\rightarrow\infty$, and such that whenever  $\lambda$ is a $d_{\G}$-geodesic segment  lying outside $B_{\G}(1; N)$ in $\Gr \G$,  the  $\HHH^3$-geodesic   $[j_n( \lambda)]$ lies outside $B_{\HHH}(O; f'(N))$  for all $n \in \mathbb N$.
\end{definition}

\begin{prop}[\cite{mahan-series1} Proposition 5.3] \label{uep}
Suppose that a sequence of  discrete faithful weakly type preserving representations $(\rho_n \co  \G \rightarrow \PSL )$
converges algebraically to  $  \rho_\infty $.  Then $(\rho_n )$ converges strongly  if and only if   it satisfies UEP.
\end{prop}

\begin{theorem}[\cite{mahan-series1} Theorem 5.7]\label{unifcrit1}  Let $\Gamma$ be a geometrically finite Kleinian group  and let $\rho_n: \G \to G_n$ be  weakly type preserving  isomorphisms to Kleinian groups. 
Suppose that $\rho_n$ converges algebraically to a 
representation $\rho_{\infty}$. Then  if   $(\rho_n)$ satisfies UEPP, the  $CT$-maps  $\hat i_n\co  \Lambda_{\G} \to \Lambda_{n}$
converge uniformly to $\hat i_{\infty}$. If  $\Gamma$ is non-elementary,  the converse also holds.
\end{theorem}

Inverting the function $f'$ in the definition of UEPP gives a slight modification of the criterion in Theorem~\ref{unifcrit1}  convenient to our purposes. Note that in the statement which follows, we do not need to assume that $f_1(N) \to \infty$ as $N \to \infty$.

\begin{cor} \label{unifcrit1a}  Let $\Gamma$ be a geometrically finite Kleinian group  and let $\rho_n: \G \to G_n$ be  weakly type preserving  isomorphisms to Kleinian groups. 
Suppose that $\rho_n$ converges algebraically to a 
representation $\rho_{\infty}$ and  that there exists a function $f_1 \co \mathbb N \to \mathbb N$ such that for any $L \in \mathbb N$, whenever  $\lambda$ is a $d_{\G}$-geodesic segment  lying outside $B_{\G}(1; f_1(L))$ in $\Gr \G$,  the  $\HHH^3$-geodesic   $[j_n( \lambda)]$ lies outside $B_{\HHH}(O; L)$  for all $n \in \mathbb N$.
Then   the $CT$-maps  $\hat i_n\co  \Lambda_{\G} \to \Lambda_{n}$
converge uniformly to $\hat i_{\infty}$.  \end{cor}
\begin{proof} In view of Theorem~\ref{unifcrit1}, it is enough to see the given condition   implies UEPP. We do this by inverting the function $f_1$. Precisely, say $L' > L$. If  $\lambda$ is a $d_{\G}$-geodesic segment  lying outside $  B_{\G}(1; f_1(L'))$ in $\Gr \G$,    the  $\HHH^3$-geodesic   $[j_n( \lambda)]$ lies outside $B_{\HHH}(O;L')$ and hence certainly outside $B_{\HHH}(O;L)$. Thus we may modify $f_1$ if needed so that 
$ f_1(L') > f_1(L)$. Hence without loss of generality using an inductive argument we may assume that $f_1$ is strictly increasing, so that in particular $f_1(N) \to \infty$ as $n \to \infty$.

Now given $N \in \mathbb N$ define $f'(N) = L$, where $L $ is the unique positive integer  such that $N \in [f_1(L), f_1(L+1))$.  Note that with this definition, $f'(N) \to \infty$ as $N \to \infty$.

Suppose given $N$ and that   $f'(N) = L$. We are given that whenever 
 $\lambda$ is a $d_{\G}$-geodesic segment  lying outside $  B_{\G}(1; f_1(L))$ in $\Gr \G$,    the  $\HHH^3$-geodesic   $[j_n( \lambda)]$ lies outside $B_{\HHH}(O;L)$  for all $n \in \mathbb N$.  Thus whenever 
 $\lambda$ is a $d_{\G}$-geodesic segment  lying outside $  B_{\G}(1; N)$ in $\Gr \G$, it also lies outside $  B_{\G}(1;f_1(L))$ and hence   the  $\HHH^3$-geodesic   $[j_n( \lambda)]$ lies outside $B_{\HHH}(O;L)  =B_{\HHH}(O;f'(N)) $ for all $n \in \mathbb N$.
   This is exactly the condition UEPP.
 \end{proof}

Applying a similar inversion  to Theorem~\ref{crit1}  we obtain immediately:
\begin{cor} \label{unifcrit1aa}  Let $\Gamma$ be a geometrically finite Kleinian group  and let $\rho \co  \G \to G$ be a weakly type preserving  isomorphism  to a non-elementary  Kleinian group $G$ for which a $CT$-map exists. 
Then there exists a function $f_1 \co \mathbb N \to \mathbb N$ such that for any $L \in \mathbb N$,  whenever  $\lambda$ is a $d_{\G}$-geodesic segment  lying outside $B_{\G}(1; f_1(L))$ in $\Gr \G$,  the  $\HHH^3$-geodesic   $[j( \lambda))]$ lies outside $B_{\HHH}(O; L)$. \end{cor}


\section{Strong Convergence}\label{sec:thmA}

In this section we prove Theorem~\ref{thm:strong=unif}. To set the scene, we begin with a brief discussion  in the case of strong convergence and bounded geometry, that is, when the injectivity radii of all manifolds in the sequence are uniformly bounded below.   We then turn to the general situation of Theorem~\ref{thm:strong=unif}, outlining as we go the relevant steps in the proof for a single $CT$-map as in~\cite{mahan-split}.

\subsection{The bounded geometry case}\label{sec:bndgeom}

In this section, which is not necessary for the general result, we illustrate how to use the  criteria of Section~\ref{sec:CTMaps} to prove Theorem~\ref{thm:strong=unif} in the case of a singly or doubly degenerate surface group with bounded geometry. This is essentially Miyachi's result~\cite{miyachi}. The brief sketch below is a reprise of the first author's original arguments   in \cite{mahan=jdg, mahan=ramanujan}.   This sketch  may also be useful to clarify the flow of the arguments from~\cite{mahan-split} which are the basis of our proof  of Theorem~\ref{thm:strong=unif}.

Recall that if $M$ is a hyperbolic $3$-manifold, the \emph{injectivity radius} $r(M;x)$ of $M$ at $x \in M$ is $r/2$ where $r$ is the length of the shortest loop 
based at $x$.  The manifold is said to have \emph{bounded geometry} if there exists 
$a >0$ such that $r(M;x) \geq a$ for all $x \in M$; in this case the \emph{injectivity radius} of $M$ is $\inf_{x \in M} r(M,x)$.

\begin{theorem}\label{thm:strong=unif1}  
 Let $\Gamma$ be a Fuchsian group such that $  \HHH^2/\G$ is a closed hyperbolic surface. Let $\rho_n \co \Gamma  \to G_n$ be a sequence
 of  strictly type-preserving
isomorphisms to geometrically finite groups $G_n$, which converge strongly to a singly or  doubly degenerate surface group $G_{\infty} = \rho_{\infty} (\Gamma )$.  Suppose moreover that  the injectivity radii of $M_n$ are uniformly bounded below by some $a>0$ for $n \in \mathbb N \cup \infty$. Then the sequence of $CT$-maps $\hat i_n: \Lambda_{\Gamma} \to  \Lambda_{G_n}$ converges uniformly to $\hat i_{\infty}: \Lambda_{\Gamma} \to  \Lambda_{G_{\infty}}$.
\end{theorem}

\begin{proof} It is sufficient to show that the sequence $(\rho_n)$ satisfies UEPP.
That is, we have to show  that there exists a function $f' \co \mathbb N \to \mathbb N$ such that $f'(N) \to \infty$ as $N \to \infty$ and such that whenever  $\lambda$ is a $d_{\G}$-geodesic segment  lying outside $B_{\G}(1; N)$ in $\Gr \G$,  the  $\HHH^3$-geodesic   $[j_n( \lambda)]$ lies outside $B_{\HHH}(O; f'(N))$  for all $n \in \mathbb N$.

By~\cite{mahan=jdg} Theorem 4.7, the Cannon-Thurston map exists for the group $G_{\infty}$. Hence, by Theorem~\ref{crit1},  there exists a   non-negative function  $f \co \mathbb N \to \mathbb N$, such that 
 $f(N)\rightarrow\infty$ as $N\rightarrow\infty$, and such that whenever  $\lambda$ is a $d_{\G}$-geodesic segment  lying outside $B_{\G}(1; N)$ in $\Gr \G$,  the  $\HHH^3$-geodesic joining
the endpoints of $i_{\infty}(j_{\G}(\lambda))$ lies outside $B_{\HHH}(O_G; f(N))$ in $\HHH^3$.
What we have to do is to show that the same function $f$ works uniformly for the representations $\rho_n$ for all sufficiently large $n$.  

Finding  the function $f$ in~\cite{mahan=jdg} was based on the following construction.  
For simplicity, suppose that $G_{\infty}$ is doubly degenerate; the singly degenerate case is similar. 
Let $S$ be a topological surface homeomorphic to $\HHH^2/\G$. By results of Minsky~\cite{minsky-ends, minsky-rigidity}, one can pick a sequence of pleated surface maps $S \to S_n \subset M_{\infty}= \HHH^3/G_{\infty}, n \in \mathbb Z$, such that the distance between $S_n, S_{n+1}$ is uniformly bounded above and below, and such that, with respect to the induced hyperbolic metrics on the $S_n$,  there are uniformly bounded quasi-isometries $S_n \to  S_{n+1}$. One deduces that the universal cover $\widetilde M_{\infty} $ of $M_{\infty}$ is  quasi-isometric to a `tree' of Gromov hyperbolic metric spaces. This is a Gromov hyperbolic space $\mathcal X$ equipped with a map $P$ onto a simplicial tree $\mathcal T$, which in the case of a doubly degenerate surface group can be taken to be the tree whose vertices are the integers, with a unit length edge joining $n$ to $n+1, n \in \mathbb Z$. The inverse image of each vertex is itself a 
 Gromov hyperbolic space; these spaces are uniformly properly embedded into $\mathcal X$. Moreover for each pair of adjacent vertices $v,v'$, there is a quasi-isometry between the spaces $P^{-1}(v), P^{-1}(v')$, again assumed to have quasi-isometry constants uniform over vertices $v$.

 In the present case, each space  $P^{-1}(n)$ is quasi-isometric to the universal cover of $S_n$. In particular the map $P \co P^{-1}(0) \to \mathcal X$ should be thought of as the lift to universal covers of the map $S_0 \to M_\infty$.  Since $S_0 = \HHH^2/\G$ is by assumption a closed surface, its universal cover is quasi-isometric to the Cayley graph of $\G$.  If we assume that the lift $O \in \HHH^2$ of the basepoint $ o \in S_0$  maps to the lifted basepoint $O \in \HHH^3 = \widetilde M_{\infty}$, then  we can replace   a geodesic $\lambda$ as in the statement of the theorem with a geodesic in the hyperbolic space $P^{-1}(0)$, see the comments following Theorem~\ref{crit1'}. 

We have to compare $\lambda$ with the hyperbolic geodesic $[\lambda]$ in $\til M = \HHH^3$ joining its endpoints.
 The key idea is to construct a `ladder' $\mathcal L_{\lambda} \subset \mathcal X$ by flowing $\lambda$  up the levels in $\mathcal X$ using the quasi-isometries between the levels, see~\cite{mahan=jdg} for details. By constructing a coarse Lipschitz projection from 
$\mathcal X$ to $\mathcal L_{\lambda}$, it
is shown that $\mathcal L_{\lambda}$ is uniformly quasi-convex in $\mathcal X$. That is, there exists $k>0$ such that any  $\mathcal X$-geodesic with ends in $\mathcal L_{\lambda}$ lies within distance $k$ of $\mathcal L_{\lambda}$.  A short argument given in the proof of ~\cite{mahan=jdg} Theorem 3.10  shows that this is sufficient to construct the required  function $f$  for the  manifold $M$.

This function $f$ depends on the quasi-isometry between $\widetilde M$ and $\mathcal X$, and it is not hard to see by inspection that the constants depend only on the injectivity radius of $M$.  The constant $k$ has a similar dependence.
Thus if we have a family of manifolds all of which have the same lower bound on injectivity radii,  the same function $f$ works simultaneously for all $M_n$ and the criterion of Corollary~\ref{unifcrit1a} is satisfied.
\end{proof}

The problem with directly extending this proof to the situation of unbounded geometry, is that it requires a ladder and projection whose constants depend on the injectivity radius of the whole end. This is clearly not possible in the case of unbounded geometry.  
However we note that the same methods can be extended to the case of a surface with punctures, see~\cite{mahan=ramanujan} Section 5.5.  (Essentially, this uses similar arguments about crossing horoballs  to those  in~\cite{mahan-series1}.) This gives the following result which we use in the proof of Proposition~\ref{ptwisecnvg} in Section~\ref{sec:brock}.

\begin{theorem}\label{thm:strong=unif2}  
 Let $\Gamma$ be a Fuchsian group such that $  \HHH^2/\G$ is a finite area hyperbolic surface. Let $\rho_n \co \Gamma  \to G_n$ be a sequence
 of  strictly type-preserving
isomorphisms to geometrically finite groups $G_n$, which converge strongly to a singly or   doubly degenerate surface group $G_{\infty} = \rho_{\infty} (\Gamma )$.  Suppose moreover that  the injectivity radii of $M_n$ outside cusps are uniformly bounded below for $n = 1, 2, \ldots, \infty$. Then the sequence of $CT$-maps $\hat i_n: \Lambda_{\Gamma} \to  \Lambda_{G_n}$ converges uniformly to $\hat i_{\infty}: \Lambda_{\Gamma} \to  \Lambda_{G_{\infty}}$.
\end{theorem}

 \subsection{Unbounded geometry} \label{sec:unbounded}
  
 Our proof of Theorem~\ref{thm:strong=unif} is based on  the method in~\cite{mahan-split}, which effectively verifies the condition of Theorem~\ref{crit1} for a single group. To show that UEPP holds for a sequence converging strongly to such a limit, we need to examine the argument carefully to understand the dependence of the constants on the limit group. 
We simplify by explaining the first part of the  proof in the case of surface groups, discussing extensions to  general manifolds with incompressible boundary
later.

The   proof of Theorem~\ref{thm:strong=unif} follows very closely that of the main result of \cite{mahan-split}, which can be roughly stated as:
\begin{theorem}[\cite{mahan-split} Theorem 7.1]\label{splitct}
Cannon-Thurston maps exist for simply or doubly degenerate surface Kleinian groups without cusps.
\end{theorem}

The actual result on which we base the proof of Theorem~\ref{thm:strong=unif} is ~\cite{mahan-split} Corollary 6.13, which we restate in an equivalent formulation as Lemma~\ref{lemma6.12} below. To give more insight into what is involved, we begin 
 by sketching the relevant parts of
 the proof of  Theorem~\ref{splitct}. For this we first require a brief digression on split geometry as introduced in~\cite{mahan-split}, which we do  in 
 \ref{sec:split}. In \ref{sec:thm7.1} we sketch the proof of 
 Theorem 4.3.  In \ref{sec:thmA1} we explain the main technical result we use,  Corollary~\ref{blocks1}. Finally  in \ref{sec:thmA2} we prove  
 Theorem~\ref{thm:strong=unif}.

 \subsubsection{Split geometry} \label{sec:split}
 Let $S$ be a hyperbolizable surface and let $M$ be a manifold homeomorphic to $S \times \mathcal I$ where $\mathcal I \subset \mathbb R$ is an interval either finite or infinite.
To say that $M$ has 
split geometry means, roughly speaking, that it is made by gluing together a succession (finite or infinite) of so-called \emph{split blocks}. These are blocks each of which is homeomorphic to $S \times [0,1]$, split by annuli round Margulis tubes corresponding to short curves.

A \emph{split subsurface} $S^s$ of a hyperbolic surface $S$ is a
(possibly disconnected) proper 
subsurface with boundary,   whose components are all essential and non-annular, and whose complement in $S$ is a non-empty
family of
non-homotopic  annuli which are $k$-neighborhoods of non-peripheral  geodesics on $S$.  
  Moreover  $S^s$ is required to have bounded geometry, in the sense that there exists some universal $\epsilon_0>0$ such that 
 each boundary component of $S^s$ is of length $\epsilon_0$.
  
 A \emph{split block} $B^s \subset B = S \times [0,1]$
is a topological product $S^s \times [0,1]$ for some split subsurface $S^s$ of $S$,  with the qualification that  its upper and lower boundaries are only required to be   split sub-surfaces of $S^s$.  Split blocks are glued along their boundaries to build up model manifolds in the spirit of~\cite{minsky-elc1}. A \emph{split component} is a connected component of a split block, see \cite{mahan-split} following Definition 4.11. 

Suppose $\mathfrak M$ is a  model manifold obtained by gluing finitely or infinitely many split blocks along their boundaries.  Section 4.3 in~\cite{mahan-split} introduces a so-called \emph{graph metric}  $d_{graph}$ on  the universal cover  $\widetilde {\mathfrak M}$. (This metric is denoted $d_G$ in~\cite{mahan-split}; the notation here is used to distinguish it from the word metric on the Cayley graph.) Roughly speaking $d_{graph}$ is obtained by first  electrocuting the lifts of Margulis tubes in each split block, and then by electrocuting the lifts of connected components of each split block. A (Gromov) hyperbolic manifold is said to be of \emph{split geometry}, see~\cite{mahan-split} Definition 4.31, if each split component is quasi-convex (not necessarily uniformly) in the hyperbolic metric on $\widetilde {\mathfrak M}$ and if in addition, the hyperbolic convex hull of the universal cover of any split component has uniformly bounded diameter in the graph metric $d_{graph}$. This uniform bound is called the \emph{graph quasi-convexity constant} and plays a crucial role in the discussion.

Given  a singly or doubly degenerate hyperbolic manifold $M$ whose fundamental group is a surface group,  a large part of the work in~\cite{mahan-split}  is to construct  a model manifold $  {\mathfrak M}$ of split geometry whose universal cover  $\widetilde {\mathfrak M}$ is  bi-Lipschitz homeomorphic to the universal cover $\widetilde M$ of $M$. The model $\mathfrak M$ is made  by consistently gluing 
 finitely or infinitely many split blocks $B_i$  so that $B_{i-1}$ is glued to $B_{i}$ along their  common boundary split surface $S_i$.
The existence of the sequence of split level surfaces and split blocks exiting the end
is  a consequence of the Minsky model \cite{minsky-elc1} for a degenerate end.   The detailed construction is intricate and involves a careful selection of the split level surfaces using the Minsky hierarchy, see \cite{mahan-split}  especially \S 3 and \S 4 for details.

\begin{thm}[\cite{mahan-split} Theorem 4.32]  \label{thm:splitgeom} The hyperbolic manifold corresponding to any singly or doubly degenerate surface group without accidental parabolics is bi-Lipschitz homeomorphic to a model manifold with split geometry.
\end{thm}

The part of this result which asserts \emph{uniform} graph quasi-convexity of the blocks is \cite{mahan-split} Proposition 4.23.  We note the point, key for our purposes here, that the  graph quasi-convexity constant is a combinatorial quantity which depends only on the topological convexity  of the surface defining the end, and is thus also uniform across all degenerate ends of any hyperbolic manifold defined by the same topological surface $S$.

\subsubsection{Rough sketch of Theorem~\ref{splitct}}\label{sec:thm7.1}
Let $S = \HHH^2/\G$ be a compact  hyperbolic surface  and let $\rho \co \G \to G$ be a type preserving isomorphism to  a singly or doubly degenerate surface group $G$.
 The criterion used in~\cite{mahan-split} to  prove the existence of a $CT$-map for $\rho$ is that given in Theorem~\ref{crit1}, but it will be convenient for our purposes to use the  alternative formulation of Corollary~\ref{unifcrit1aa}.  In view of Theorem~\ref{thm:splitgeom},  we may work either with $M = \HHH^3/G$,  or with a quasi-isometric model manifold of split geometry $\mathfrak M$. Lifting to universal covers, we obtain an identification of the universal cover 
 ${\widetilde{\mathfrak M}}$ of $\mathfrak M$ with $\HHH^3$, and in particular we can identify a basepoint $O \in \HHH^3$ with a point, also denoted $O$, in  ${\widetilde{\mathfrak M}}$.

 Here is the statement we need:
 \begin{prop}  \label{uepp}
Let $G$ be  a totally degenerate surface group corresponding to a strictly type preserving representation $\rho\co \pi_1(S) \to G$ where $S$ is a closed surface as above,  and let $\mathfrak M$ be a model manifold of split geometry    corresponding to  $M = \HHH^3/G$. Let $B_0 \subset S_0 \times [0,1]$ be a fixed base block and let $\phi \co  S \to \mathfrak M$ be the embedding which identifies  $S$ with $S_0 \times \{0\} \subset B_0$.
 Fix a  basepoint  $O \in \widetilde S $  in the universal cover  $ {\widetilde{S} }= \HHH^2 $. Denote by $\til \phi$  the lift of $\phi$ such that $\til \phi(O)$ is the basepoint  $O \in \HHH^3 = {\widetilde{\mathfrak M}}$.
 
Then for any $L\in \mathbb N$, there exists    $f(L) \in \mathbb N$, such that whenever   $\lambda$ is a geodesic segment in $(\widetilde S, d_S)$ lying outside an $f(L)$-ball
around ${O}\in{\widetilde{S}}$ (where $d_S$ is the lifted hyperbolic metric on $\til S$), the hyperbolic geodesic $[\til\phi (\lambda)]$ in $\widetilde {\mathfrak M} $  joining the endpoints of $\til\phi(\lambda)$   lies 
outside the $L$-ball around 
${O}\in \widetilde{{\mathfrak M}}$.
 \end{prop}

Proposition~\ref{uepp} follows from~\cite{mahan-split} Lemma 6.12, restated as Lemma~\ref{lemma6.12} below. Theorem~\ref{splitct} follows immediately from Proposition~\ref{uepp} on applying Corollary~\ref{unifcrit1aa}. The condition that $G_\infty$ be totally degenerate is introduced simply to avoid the annoyance of having to deal with geometrically finite ends.

  \begin{rmk} {\rm We refer to the discussion in Section~\ref{sec:CTMaps} for the equivalence of the condition as stated here with the condition on geodesics in the Cayley  graph $\Gr \G$.}
 \end{rmk}

 \begin{rmk} {\rm Strictly speaking,  since $S_0$ is $S$ with some Margulis tubes deleted, it cannot be identified with $S$. For a precise statement we need to work instead with \emph{welded split blocks} in which the  ends of the Margulis tubes deleted in the split blocks $B_i$  are reglued. Gluing the welded split blocks along their boundaries gives a \emph{welded model manifold}   $ {\mathfrak M}_{wel}$ homeomorphic to $S \times \mathbb R$ or $S \times [0,\infty)$ according as $M$ is doubly or singly degenerate,   see \cite{mahan-split} \S 4.3 for details. For simplicity, we will ignore this distinction in the discussion below.}
 \end{rmk}

  The essence of  Proposition~\ref{uepp}  is contained in Lemma~\ref{lemma6.12} below, whose proof
 occupies \S 5 and \S 6 of~\cite{mahan-split}.  As in the bounded geometry case, the first step is to use $\lambda$ (or more precisely, the image $\til \phi(\lambda)$ of $\lambda$ in $\widetilde {\mathfrak M}$) to construct a `ladder' $\mathcal L_{\lambda}$ by `flowing up' through the blocks of $\widetilde {\mathfrak M}$.  By constructing a coarse retraction onto  $\mathcal L_{\lambda}$, it  is shown (\cite{mahan-split} Corollary 5.8) to be quasi-convex in the graph metric $d_{graph}$. The constants here are independent not only of $\lambda$, but also of the particular model $\mathfrak M$,   depending in fact only on the topological type of the surface $S$.  Starting from $\lambda$, this allows one to construct  an `admissible' path   joining the endpoints of $\lambda$ which follows the levels $\mathcal L_{\lambda}$ and which, after a controlled sequence of alterations using a succession of different metrics, is ultimately modified into an electro-ambient quasi-geodesic   in $(\widetilde {\mathfrak M}, d_{CH})$ joining the endpoints of $\lambda$. 
 Here  $d_{CH}$ is the metric on $\widetilde  {\mathfrak M}$ obtained by electrocuting the hyperbolic convex hulls $CH(\til K)$ of the extended split components $\til K$ of $\widetilde  {\mathfrak M}$.  (Recall that an electro-ambient quasi-geodesic  in $(\widetilde {\mathfrak M}, d_{CH})$  is an electric quasi-geodesic  whose intersection with each electrocuted component    is geodesic in the hyperbolic (or model) metric, see Section~\ref{sec:relhyp}.)   
 This finally leads to the following key statement, which we once again formulate in the modified form appropriate to Corollary~\ref{unifcrit1aa}:
 
 \begin{lemma} [\cite{mahan-split} Lemma 6.12] \label{lemma6.12}
Let $ {\mathfrak M}$ be a model manifold of split geometry  and without cusps for a  fixed compact hyperbolizable surface
 $S$. Let $B_0 \subset S_0 \times [0,1]$ be a fixed base block and let $\phi \co S \to {\mathfrak M}$ be an embedding   which identifies  $S$ with $S_0 \times \{0\} \subset B_0$.
 Let $o \in S_0 \subset  {\mathfrak M}$ be a  basepoint and fix a lift $ O $ in the universal cover $ \widetilde S  = \HHH^2 $. Let $\phi $ be the lift $\til \phi \co \til S \to  \til  {\mathfrak M} =  \HHH^3$ so that $\til \phi(O) = O$ as in  Proposition~\ref{uepp}. 
  
Then for any $L>0$, there exists $f (L) >0$, such that for any geodesic segment   $\lambda$ in $(\widetilde S, d_S)$ lying outside $B_{\HHH^2}(O,f (L)) \subset {\widetilde{S}}$, there is an electro-ambient quasi-geodesic  $\gamma$ in $(\widetilde {\mathfrak M}, d_{CH})$  joining the endpoints of $\til \phi(\lambda)$  which lies 
outside   $B_{\HHH^3}(O, L) \subset  \widetilde{{\mathfrak M}}$.
 \end{lemma}

To get from Lemma~\ref{lemma6.12} to the statement of  Proposition~\ref{uepp}, we have now to only to show that  one can replace electro-ambient quasi-geodesic  $\gamma$ in $(\widetilde {\mathfrak M}, d_{CH})$  by the   hyperbolic geodesic $[i(\lambda)]$  joining the endpoints of $\lambda$.  This can be done in virtue of Lemma~\ref{ea-genl}, see also \cite{mahan-split} Lemma 2.5 and the short argument  in the proof of \cite{mahan-split} Theorem 7.1.

 \subsubsection{Proof of Theorem~\ref{thm:strong=unif}: preliminaries}  \label{sec:thmA1}
 
 We  want to adapt the above discussion to the situation of Theorem~\ref{thm:strong=unif}, namely a sequence of geometrically finite purely loxodromic groups $G_n$ converging strongly to a degenerate group $G_{\infty}$.

    First, suppose that we have a single  type preserving discrete faithful representation $\rho \co \Gamma \to G$ where $ \Gamma$ is geometrically finite purely loxodromic group and $G$ is a totally degenerate   Kleinian group with incompressible ends.  As discussed in \S 4.7 of \cite{mahan-split}, Theorem~\ref{thm:splitgeom} extends to any simply degenerate end of any hyperbolic manifold with incompressible ends.  Thus one can easily modify Lemma~\ref{lemma6.12} to apply in this more general situation.

 The essence of the proof  of Theorem~\ref{thm:strong=unif}   is therefore to understand the dependence of the constants involved on the manifold $M = \HHH^3/G$. The manifold, and hence the   ladder $\mathcal L_{\lambda}$ (really one ladder for each end of $M$), may not have bounded geometry, since its construction depends on the whole of each infinite end  of $M$. However as already noted, the Lipschitz constant for the coarse projection to $\mathcal L_{\lambda}$ in the graph metric $d_{graph}$ depends only on the topological type of the surfaces defining the ends and is thus uniform over any approximating sequence $\rho_n$. 
    
  Next, we need to look into the effect on the constants involved of the modifications needed to get from the   $d_{graph}$-electro-ambient geodesic  joining the endpoints of $\til \phi(\lambda)$ as in~\cite{mahan-split} Section 6.1 (called $\beta_e$ in~\cite{mahan-split})  to a 
   $d_{CH}$-electro-ambient geodesic with the same endpoints as in
   Lemma~\ref{lemma6.12} and  hence to a hyperbolic geodesic as in Proposition~\ref{uepp}.  
Suppose we have a bounded segment $\gamma$ of  such an electro-ambient geodesic which passes through a given finite collection of split blocks. Close examination of the proof in \cite{mahan-split} shows that the modifications made in the course of replacing $\gamma$ with the corresponding segment $\hat \gamma$ of the required hyperbolic geodesic, depend only  on
the geometry of the hyperbolic convex hulls of the split components traversed by $\gamma$.  The number $N(\gamma)$ of convex hulls traversed  is bounded uniformly in terms of the graph quasi-convexity constant. Thus the modifications made to $\gamma$, which control how much nearer $\hat \gamma$  approaches the  basepoint than $\gamma$, depend only the geometry of $N(\gamma)$ blocks, where $N(\gamma)$ depends only on the topology of the ends of $M$.

 This leads to  an equivalent reformulation of Lemma~\ref{lemma6.12}, which is essentially the same as~\cite{mahan-split}  Corollary 6.13 in the context of a general  hyperbolic manifold $M$ without cusps. We begin with some more notation. Suppose $M$ has  ends $E^1, \cdots , E^r$, and  that each $E^k$ is homeomorphic to $S^k \times [0,\infty)$ for some closed hyperbolic surface $S^k$, and such that each end is simply degenerate.    Let $\mathcal  K \hookrightarrow  M$ be a Scott core cutting off  
the ends $E^k$, so that the boundary components of $\mathcal  K$ are the surfaces $S^{k} \times \{0\}, i=1, \ldots , {r}$.   
 Theorem~\ref{splitct} asserts that  each end $E= E^k$  of $M$ has split geometry and has a model   made by consistently gluing 
  split blocks $B_i = B^k_i, i \in \mathbb N$,  so that $B_{i-1}$ is glued to $B_{i}$ along their  common boundary split surface $S_i = S^k_i$.  For $q \in \mathbb N$, let $\BB  (q) = \mathcal  K \cup  \bigcup_{k=0}^r\bigcup_{i=0}^qB^k_i$  be the manifold formed by gluing the core $\mathcal K$ to the first $q$ blocks in each end  as above. Replacing  models by actual ends, we may assume that $\BB  (q)$ is quasi-isometrically embedded in $M$.
  
We also change our formulation so as to be in accordance with the criterion in terms of $\Gr \G$, see the explanation in Section~\ref{sec:CTMaps}.   Given an isomorphism $\rho \co \G \to G$, we have an embedding   $j  \co \Gr \G \to \HHH^3,   j(\gamma) = \rho(\gamma) \cdot O$ 
of the Cayley graph  of $\G$ into $\HHH^3$.
Recall also that if $\lambda$ is any $d_{\G}$-geodesic segment in $\Gr \G$  with endpoints $\gamma, \gamma' \in  \G$, then  we write $j(\lambda)]$ for  the $\HHH^3$-geodesic $[j(\gamma),j( \gamma')]$.

\begin{rmk} {\rm We also have the map $i \co j_{\G}(\Gr \G) \to \HHH^3  $  defined by $ i  j_{\G}(\gamma) = j(\gamma) $, whose extension to $\Lambda_{\G}$ is the $CT$-map $\hat i$. Note that  $i$  is morally the same as the embedding $\til \phi \co \til S \to \til {\mathfrak M}$ of Proposition~\ref{uepp}, and as long as $S$ is closed, the map $j_{\G}$ is a quasi-isometry so that $i$ is also morally equivalent to $j$.
} \end{rmk}

 \begin{prop} \label{blocks} Let $\G$ be a geometrically finite convex cocompact Kleinian group and let $\rho \co \G \to G$ be a strictly type preserving isomorphism.  Suppose that 
  $ M = \HHH^3/G$ is a hyperbolic $3$-manifold without cusps,  with Scott core $\mathcal  K$ and incompressible ends $E^1, \cdots , E^r$ as above.  Then there  exists   $D \in \mathbb N$, depending only on the topology of the ends of $M$,   with the following property. 
Suppose given $ q \in \mathbb N$ and that the submanifold $\BB  (q)$ is defined as above.   
Then for all $L>0$ 
 there exists   $f(L) >0$, depending only on the geometry of the submanifold $\BB  (q+D)$, 
 such that if $\lambda$ is any $d_{\G}$-geodesic segment in $\Gr \G$ which lies outside $B_{\G}(1, f(L))$, then $[j(\lambda)] \cap \widetilde { \BB(q)}$ lies outside $B(O,L)$, 
 where $ \widetilde { \BB(q)}$ is the lift to $\til M$ of $ { \BB(q)}$.
 \end{prop}

Thus the proposition asserts that, independent of the geometry of $M$ outside $ \BB(q+D)$,   we can control 
$[j(\lambda)]$ inside $ \widetilde{ \BB(q)}$ with constants which depend  only on the first $q+D$ blocks in each end, where $D$ is a universal constant which depends only on the topological types of the ends. 
 
We immediately deduce the following Corollary, which will be used in the proof of Theorem~\ref{thm:strong=unif}.

\begin{cor} \label{blocks1} Let $M, \mathcal K,  D$ be as in  Proposition~\ref{blocks}. Let $q \in \mathbb N$. Then  there exists   $f \co \mathbb N \to \mathbb N$ with the following property.
  Suppose that $V$ is any hyperbolic manifold  such that  there is a bi-Lipschitz embedding $\beta \co \BB(q+D) \to V$, and such  that the ends of  $V \setminus \beta(\BB(q+D)) $  have split geometry and correspond bijectively and homeomorphically to the ends of $M \setminus \BB(q+D)$.   Let  $\tilde \beta \co \widetilde {\mathcal K} \to \widetilde {V}$  be the lift of $\beta$  to the universal covers.  Suppose that $L>0$ and that   
 $\lambda$ is any $d_{\G}$-geodesic segment in $\Gr \G$ which lies outside $B_{\G}(1, f(L))$.   Then $[\til  \beta(j(\lambda))] \cap \til  \beta(\til  \BB(q))$ lies outside $B(\beta(O), L) \subset \widetilde V$, where $[\tilde \beta(j(\lambda))]$ denotes the geodesic whose endpoints are the images under $\tilde \beta \circ j$ of the endpoints  of $\lambda$. \end{cor}

\subsubsection{Proof of Theorem~\ref{thm:strong=unif}: conclusion } \label{sec:thmA2}

We have a  geometrically finite   Kleinian group  $\Gamma$  without parabolics,   which does not split as a free product, together with a sequence of strictly type preserving
isomorphisms  $\rho_n \co \Gamma  \to G_n$  which converge strongly to a  purely loxodromic Kleinian group $G_{\infty} = \rho_{\infty} (\Gamma )$.
Our aim is to use the criterion of Corollary~\ref{unifcrit1aa} to show that the corresponding sequence of $CT$-maps $\hat i_n: \Lambda_{\Gamma} \to  \Lambda_{G_n}$ converges uniformly to $\hat i_{\infty}: \Lambda_{\Gamma} \to  \Lambda_{G_{\infty}}$.

Since  the representations $\rho_n$ converge strongly to $G_\infty$, 
by Proposition~\ref{cores}  we may as well assume that we have a compact core 
$K$ of $N = \HHH^3/\Gamma$ together with
embeddings $\phi_n \co K \hookrightarrow M_n = \HHH^3/G_n, n \in \mathbb N \cup \infty$, such that $K_n = \phi_n(K)  $ is a Scott  core  of  $M_n$ and  $\mathcal K = \phi_{\infty}(K)$ is a Scott core of $M_{\infty}$.  We may also choose the lift $O$ of the  basepoint $o \in K$  and lifts $\til \phi_n$ of $\phi_n$ so that for all $n$, 
$\tilde \phi_n (O)$   lies in a uniformly bounded neighborhood of $O \in \Hyp^3$. 

We need   to show that there exists a  function  $f_1 \co \mathbb N \to \mathbb N$ such that whenever  $\lambda$ is a $d_{\G}$-geodesic segment  lying outside $B_{\G}(1; f_1(L))$ in $\Gr \G$,  the  $\HHH^3$-geodesic   $[j_n( \lambda)]$ lies outside $B_{\HHH}(O;L)$  for all $n \in \mathbb N$, where as usual $j_n \co \Gr \G \to \HHH^3, j_n(\gamma) = \rho_n(\gamma) \cdot O$.

As in \ref{sec:thmA1} above, let $E^i$, $i=1, \ldots, r$ be the ends of $M= M_{\infty}$. 
Given $L \in \mathbb N$, choose compact submanifolds  $E^i_1 \subset E^i$ homeomorphic to $S^i \times [0,1]$
such that $d_M((E^i \setminus E^i_1), \mathcal  K) \geq 2L$, where $d_M$ denotes the metric induced by shortest paths in $M$.  Identifying $M$ with the model $\mathfrak M$, pick $q \in \mathbb N$ so that $E^i_1 $ is contained in the union of the first $q$ blocks of $E^i$ for each $i= 1, \ldots, r$.  Then  with the notation of Proposition~\ref{blocks}, we have 
$d_M(M \setminus \BB(q), \mathcal  K) \geq 2L$.

By strong convergence, there exists $n_0 =n_0(L) \in \mathbb N$ such that for all $n \geq n_0$, 
there exists a $2$-bi-Lipschitz embedding $\psi_n \co \BB(q+D)  \rightarrow M_n$, and such that     $\psi_n ( \BB(q+D) )  \supset K_n = \psi_n(\K)$ cuts off the ends of $M_n$, see Lemma~\ref{cutoffend}.  Let $o \in \K$ be a base-point which lifts to  $O \in \HHH^3$.
Up to adjusting by a bounded distance, we may assume that $
\psi_n$ lifts to a map $\tilde \psi_n$ with $\tilde \psi_n (O) = O$, and that $\tilde \psi_n (O) $ projects to the base-point $o_n \in M_n$ and $\psi_n(o) = o_n$.  
 
Now apply Corollary \ref{blocks1}  with $M =M_{\infty} $ and the integer $q$, and with $V=M_n$ and $\beta = \psi_n$, to see  there exists a function $f  \co \mathbb N \to \mathbb N$, independent of $n$,  with the following property. Let  $\lambda$ be a $d_{\G}$-geodesic segment in $\Gr \G$   lying outside $B_{\G}(1, f (2L))$.
Since $(\psi_n)_* \circ \rho_{\infty} = \rho_n$ we have
$\til \psi_n \circ j_\infty = j_n$.  Let  $ [j_n(\lambda)]$ be the $\HHH^3$-geodesic segment  with the same endpoints as $ j_n(\lambda)$.
Then  by  Corollary \ref{blocks1}, $ [j_n(\lambda)] \cap \til \psi_n(  \til {\BB(q)})$ lies outside $B_{\HHH^3}(O,2 L) \subset \HHH^3$, where $   \til {\BB(q)} \subset \til M_{\infty}$ is the lift to $\til M_{\infty}= \HHH^3$ of  $ \BB(q)$.

Let $\pi\co \HHH^3 \to   M_n$ be the covering projection so that in particular $\pi(O) = o_n $. Since  $  \psi_n(  \BB(q))$ cuts off the ends of $M_n$, then since $\psi_n$ is 2-bi-Lipschitz,  any point in $\pi( [j_n(\lambda)]) $ which lies outside $\psi_n(  \BB(q))$
must be at least distance $L$ to $o_n$. It follows that  whenever $n \geq n_0$, 
$ [j_n(\lambda)]$ lies outside  $B(  \til \psi_n(O),L) \subset \til M_n = \HHH^3$  whenever $\lambda$  lies outside  $B_{\G}(1,f(L))$ in $ {\Gr \G}$.

It remains only to deal with $n < n_0$. By Corollary~\ref{unifcrit1aa}, for each $n \in \{ 1, \ldots, n_0\}$,
there exists $N_n = N_n(L)  \in \mathbb N$, such that if 
$\lambda$   is a $d_G$-geodesic outside $B_{\G}(1, N_n(L))$  in $ {\Gr \G}$, then $[j_n(\lambda)]$ lies outside $B (O, L) \subset \HHH^3$.
Choosing $f_1(L) = \max \{f(L), N_1(L), \ldots ,N_{n_0}(L)\}$ we have verified the criterion of Corollary~\ref{unifcrit1aa}. This completes the proof of Theorem~\ref{thm:strong=unif}.

\section{Algebraic limits and non-convergence of limit points}\label{sec:brock}

In this section we prove Theorem~\ref{brockexample}, of which Theorem~\ref{thm:alg=ptwise} is an immediate consequence. The sequence of groups $G_n$ in Theorem~\ref{brockexample} is that described by Brock in \cite{brock-itn}, in which the convergence is algebraic but not strong. We begin with
a brief description of  the examples and Brock's bi-Lipschitz models for the  manifolds involved.

\subsection{Brock's Examples}  \label{brockexamples}
   The groups $G_n$ in Theorem~\ref{brockexample} are
a sequence of quasi-Fuchsian surface groups converging
algebraically but not strongly to a partially degenerate geometrically infinite surface group $G_{\infty}$ with an accidental parabolic. 
The examples are also discussed briefly in~\cite{minsky-cdm}.

The sequence $G_n$ is obtained as follows.
Fix a closed hyperbolic surface $X= \HHH^2/\G$. Let $\sigma$ be  a simple closed geodesic which separates  $X$ into two  subsurfaces $R$ and $L$.  Let $\alpha$ denote an automorphism of $X$ such that $\alpha |_{L}$ is the identity and $\alpha  |_{R} = \chi $ is a pseudo-Anosov diffeomorphism of
$R$ preserving the boundary $\sigma$. (For later reference, it is important to ensure that  there is no  Dehn twisting around $\sigma$ when  $\chi$ is considered as the restriction of $\alpha$ to $\pi_1(R)$, see~\ref{sec:noncnvg} below.)
Let $G_n$ be the quasi-Fuchsian group given by the simultaneous uniformization of $( \alpha ^n(X),X)$, so that $G_n = \rho_n(\G)$ for suitably normalized $\rho_n : \G \to SL(2,\C)$ and $G_0$ is Fuchsian.  
This means that  the regular set $\Omega_n$ of $G_n$ has two components 
$\Omega_n^{\pm}$ where the `lower' component $\Omega_n^{-}/\G$  is conformally equivalent to $X$ and  the `upper' component   $\Omega_n^{+}/\G$
is equivalent to $\alpha ^n(X)$.
The algebraic limit $G_{\infty}$ of the  groups $G_n$ is a partially degenerate geometrically infinite surface group, while with suitable choice of basepoint, the geometric limit of the  manifolds $M_n= \HHH^3/G_n$ is homeomorphic to $X\times \mathbb R \setminus R \times \{0\}$. These assertions will be explained  in more detail below.  Since to 
 fully understand our example, it is important  to be clear about the notational conventions, we begin by setting these out. As far as possible, we follow Brock  \cite{brock-itn}.

 \subsubsection{Teichm\"uller space and the mapping class group} \label{sec:teich}

Let $S$ be an oriented hyperbolizable surface. The  Teichm\"uller space  $\Teich(S)$ of $S$ parametrizes finite area hyperbolic structures on $S$ up to isotopy. Thus a point $X \in \Teich(S)$ is a hyperbolic surface $X$ equipped with a homeomorphism $f \co S \to X$ which marks $X$.
 The mapping class group $\Mod (S)$ acts on $\Teich(S)$: 
if $\alpha \in \Mod (S)$ then $\phi(S,f) = (S, f \circ \alpha^{-1})$. The map on surfaces induces an action on the space of representations $\rho \co \pi_1(S,x_0) \to \PSL$ by 
$\alpha (\rho ) = \rho \circ \alpha^{-1}$.

 If $\zeta \co [0,1] \to S$ is a  path in $S$, we denote by $ \alpha(\zeta)$ the path  $\alpha \circ \zeta   \co [0,1] \to S$.
Thus defining a hyperbolic surface $X$ in terms of the lengths $\ell_X(\gamma)$ of the simple closed curves $\gamma$ on $X$ and identifying $X$ with $S$ by taking $f = \rm{id}$, we can write
\begin{equation} \label{eqn:lengths}
\ell_{\alpha(X) }(\gamma) =\ell_X(\alpha^{-1}(\gamma) ). \end{equation}
Hence the map $\alpha\co X \to\alpha(X)$ is an isometry.
The action on curves extends to an action on the space of measured laminations $\ML(S)$ on $S$: for $\mu \in \ML(S)$ the lamination $\alpha(\mu)$ is defined by (geometric) intersection numbers:
\begin{equation}\label{eqn:lengths1}
i(\gamma, \alpha(\mu) )=  i(\alpha^{-1}(\gamma), \mu) \end{equation} for all $\gamma \in \pi_1(S)$.
Thus if $\alpha \in \Mod(S), \mu \in \ML(S)$ and $X \in \Teich(S)$ we have
$
\ell_{\alpha(X)} (\mu) = \ell_{X} (\alpha^{-1}(\nu)) $
 where $\ell_X(\mu)$ denotes the length of the lamination $\mu$ in the surface $X$. 

 \subsubsection{Quasi-Fuchsian groups}  \label{sec:quasifuchs}

A quasi-Fuchsian group $G$ is the image of a discrete faithful representation $\rho \co \pi_1(S,x_0) \to \PSL$, whose
 domain of discontinuity has two simply connected components $\Omega^{\pm}$.  
 By convention we take $\Omega^+$ to be the component whose orientation is the same as that of $S$.
By Bers' simultaneous uniformisation theorem, a pair of points $X, Y \in \Teich(S)$ parametrize quasi-Fuchsian groups: if $G=G (X,Y) = \rho(\pi_1(S,x_0))$ then $Q(X, Y)$ denotes the manifold 
 for which
$\Omega^+/G $ is conformally equivalent to $X$ and $\Omega^-/G $ is anti-conformally equivalent to $Y$.
Strictly this only defines $G(X, Y)$ up to conjugation. 

To mark $Q(X, Y)$  and fix $G(X, Y)$ we proceed as follows. 
  Fix a base point $s_0 \in S$. Let $Y $ be the point $(S,f) \in \Teich(S)$ and let $ y_0 = f(s_0)$.
Fix a lift  $\tilde f \co  \HHH^2  \to \Omega^-$   which descends to $f$, where we identify $\HHH^2 $ with the universal cover of $S$ and where $O= O_2 \in \HHH^2$ is a fixed basepoint
which descends to $s_0$, and set $\tilde y_0 = \tilde f (O_2)$.
Now the convex hull $\CC$ of $Q(X, Y)$ is bounded by two pleated surfaces $\partial \CC^{\pm}$ which (when positively oriented, that is, so that
$\partial \CC^{+}$ is oriented pointing out of $\CC$ and  $\partial \CC^{-}$ is oriented pointing into $\CC$)
are respectively uniformly bounded distance to $X, Y$ in $\Teich(S)$. There is a natural retraction map $ r$ from $ \Omega^-$ to the lift $\widetilde {\partial \CC^-}$  to $\HHH^3$ of $\partial \CC^-$~\cite{EpM}; we arrange that $r(\tilde y_0) = O = O_3 \in \HHH^3$.  This gives a map 
$\tilde f \co (\HHH^2, O_2) \to (\widetilde {\partial \CC^-}, O_3)$ which descends to a map $S \to Q(X,Y)$ which sends $S$ to a pleated surface in $Q(X,Y)$ at uniformly bounded Teichm\"uller distance to $Y$.  We use this marking to induce  the representation $ \rho \co \pi_1(S, s_0) \to  G(X, Y)$.
We fix the basepoint $o \in Q = \HHH^3 /G$ to be the projection of the point $O \in \HHH^3$.

 \subsubsection{Iteration of pseudo-Anosovs} \label{sec:pseudo} For details on measured laminations and pseudo-Anosov maps, see \cite{FLP, Otal}.  Here is a summary of what we need. Let $\chi \in \Mod (S)$ be pseudo-Anosov. Then $\chi$ has two fixed points in the space $PML(S)$ of projective measured laminations on $S$:
 the   \emph{stable lamination} $\lambda^s$ and an \emph{unstable lamination} $\lambda^{u}$.  
(If necessary, we distinguish the underlying lamination $|\lambda|$ from its transverse measure $\lambda$.) This means there exists $c>1$ so that  $\chi(\lambda^s) =  \dfrac{1}{c}\lambda^s$ and $\chi(\lambda^u) =  {c}\lambda^u$. 
From \eqref{eqn:lengths1} this gives $ i(\chi^{-1}(\gamma), \lambda^u)  = i(\gamma, \chi(\lambda^u))  = ci(\gamma, \lambda^u)$.  
Thus $\lim_{n\to \infty} i(\chi^{n}(\gamma), \lambda^u) \to 0$ so that (since $\lambda^u$ is uniquely ergodic), 
\begin{equation}\label{eqn:cnvg}   [\chi^{n}(\gamma)] \to [\lambda^u]  
\ \ \rm{in} \    PML(S) 
\end{equation}
 where $[\mu]$ denotes the projective equivalence class of $\mu \in ML(S)$ in $PML(S)$.

 Let $X$ be a fixed surface in $\Teich(S)$. 
The boundary leaves of laminations $|\lambda^u|,| \lambda^s|$ decompose $X$ into a collection of rectangles in each of which we have a  metric $\sqrt {(\lambda^u)^2 +( \lambda^s)^2}$ (or more simply the equivalent metric $ \lambda^u+\lambda^s $). Putting these together gives a metric quasi-isometric to the hyperbolic metric on $X$.
Then $\ell_X(\gamma) \sim i(\gamma, \lambda^u) +  i(\gamma, \lambda^s) $ and for any arc $T$ we have $\ell_X(T) \sim i(T, \lambda^u) +  i(T, \lambda^s) $, where $\sim$ denotes equality up to multiplicative bounded constants. In particular, if $T$ is an arc along an unstable leaf then $\ell_X(T) \sim   i(T, \lambda^s) $ so that 
$\ell_{\chi(X)}(T) \sim   i(\chi^{-1}(T), \lambda^s)  = c^{-1}  i( T, \lambda^s) \sim c^{-1}\ell_ X(T) $.  In other words, 
$\phi$ contracts along unstable leaves.
Also observe that  since 
\begin{equation} \label{lengths2}
\ell_{ \chi^n(X)}(\gamma) = \ell_{  X}(\chi^{-n}(\gamma)) \sim    i( \gamma, \chi^{n}\lambda^s) +i( \gamma, \chi^{n}\lambda^u)    = c^{n} i( \gamma,  \lambda^u) + c^{-n} i( \gamma,  \lambda^s),\end{equation}
 it follows by taking ratios of lengths  that $ \chi^n(X) \to [\lambda^u]$ in $PML(S)$, viewed as the Thurston compactification 
of $\Teich(S)$.

 \subsubsection{The algebraic limit for iteration of pseudo-Anosov maps}  \label{sec:pseudoAlimit} 

Given a surface $S$ and $ \alpha \in \Mod(S)$, the \emph{mapping torus} of $(S,\alpha)$ is the manifold $T_{\alpha} = S \times [0,1] / \sim $   where $\sim $  is the equivalence relation $(x,0) \sim ( \alpha(x),1)$.    
Let  $N_{\alpha} = S \times \mathbb R $ be the cyclic cover of $T_{\alpha}$, corresponding to the subgroup $\pi_1(S)$.   The manifold $N_{\alpha}$ is naturally oriented by the orientation of $S$.
  
If $\chi \in \Mod(S)$ is pseudo-Anosov, then Thurston showed that  $T_{\chi}$  and hence $N_{\chi}$ has a hyperbolic structure~\cite{Otal, ctm-renorm}. 
Pick $(\Sigma,f) \in \Teich(S)$ and consider the  manifold $M_{\chi}$ corresponding to the 
algebraic limit of the quasi-Fuchsian groups $G(\chi^n (\Sigma),\Sigma)$, where $Q(\chi^n (\Sigma),\Sigma)$  is marked as described above. 
McMullen
~\cite{ctm-renorm} Theorem 3.11 shows that the  limit manifold   
 $M_{\chi}$ has  one positive degenerate end $E$ which is asymptotically isometric to 
the positive end of  $N_{\chi }$. (Note that the positive end of $N_{\chi }$ is, up to complex conjugation, the negative end of $N_{\chi^{-1} }$, see ~\cite{ctm-renorm} Proposition 3.10.)
 
The end  $E = E(\Sigma,\chi)$ 
consists of successive sheets 
  which are mapped one to the next by $\chi $.
  More precisely,  we have a sequence of pleated surfaces $h_j \co S \to E$ exiting $E$ such that the $j^{th}$ level surface is marked by the map  $   f \circ \chi^{-n} \co S \to M_{\chi}$. In particular $h_0 \co S \to \Sigma \times \{0\}$ is the map 
$ h_0(x) = (  f^{-1}(x), 0 )$ for $x \in \Sigma$; loosely, $h_0$ identifies $S$ with  $\Sigma_0  = \Sigma \times \{0\}$.
Up to quasi-isometry, $E $ is modelled by $ \Sigma \times [0,\infty)$  with the image of the
   $j^{th}$ level  pleated surface $ \Sigma_j $ identified with  $ \Sigma \times \{j\}$.  
   The hyperbolic structure of $ \Sigma_j$ is  the point $\chi^n( \Sigma) \in \Teich(S)$.
Now put a metric on  $ \Sigma \times [0,1]$  which smoothly interpolates between $ \Sigma_0$ and $ \Sigma_{1}$ and then transport this metric to 
 $ \Sigma \times [i,i+1]$  using the isometry $\chi^{i}$. This gives  a uniformly bi-Lipschitz homeomorphism from $E $ to the model manifold $  \Sigma \times [0,\infty)$. The model is marked by the map $h_0 \circ f$ which sends a base-point $s_0 \in S$ to a base-point
  $o = (f(s_0),0) \in \Sigma_0$
 
 The convex cores  of the approximating manifolds $Q(\chi^n ( \Sigma), \Sigma)$ are equally modelled by 
 $  \Sigma \times [0,n]$ with the restriction of the above metric. 
 The above marking  is, up to a uniformly bounded discrepancy, the same as the one described in~\ref{sec:quasifuchs}, and hence determines the limit representation $ \rho \co \pi_1(S,s_0) \to \pi_1(M_{\chi}, o)$.

 \begin{lemma}[\cite{brock-itn} Lemma 4.4] \label{endinglam} The ending lamination of   the end $E$ of $M_{\chi}$ is the unstable lamination $ \lambda^u$ of $\chi$.  \end{lemma}  
\begin{proof} This is proved in~\cite{brock-itn}. Here is a
variant which will serve as a check we have the correct conventions. Think  of the model $E =  \Sigma \times [0,\infty)$ as quasi-isometrically embedded in $M_{\chi}$. Let $s_0 $ be the base-point in $S$. As above, with $(\Sigma, f) \in \Teich (S)$ we  have base-point  $o= (f(s_0),0) \in  \Sigma_0 \subset M_{\chi}$. A loop $\gamma \in \pi_1(S,s_0) $ defines a path 
$(f \circ \gamma,0) \subset    \Sigma_0 $ and hence a homotopy class $[\rho(\gamma)] \in \pi_1(M_{\chi}, o)$.
 Now consider the path  $\chi^n( \gamma) = \chi^n \circ \gamma \in \pi_1(S)$. To find the approximate length of the geodesic in the class   $[\rho(\chi^n( \gamma))]$ in $M_{\chi}$,  let $\tau_n $ be the path $ t \mapsto  (f(s_0),t), t \in [0,n] \subset E$.
Note that  $ (f \circ \chi^n( \gamma), 0)  $ is homotopic   in $M_{\chi}$ to the loop $\tau_n (f \circ \chi^n( \gamma),n) \tau_n^{-1}  $ and hence freely
 homotopic to the path $ (\chi^n( \gamma),n)  \subset    \Sigma_n $.

 Now   by \eqref{eqn:lengths}, $\ell_{ \Sigma_n} (\chi^{n}(\gamma)) = \ell_{\chi^n( \Sigma)} (\chi^{n}(\gamma)) =   \ell_{ \Sigma} ( \gamma)$.
Thus the sequence of curves 
$(\chi^{n}(\gamma),n)$ on the pleated surfaces $ \Sigma_n$ exit the positive end of $M_{\chi}$ and have uniformly bounded length. This means they converge to the ending lamination of $M_{\chi}$.
On the other hand, by \eqref{eqn:cnvg}, $[\chi^{n}(\gamma)] \to   [\lambda^u]$ in $PML(\Sigma)$. Hence
the ending lamination of   $M_{\chi}$ is $ \lambda^u$. \end{proof} 

 \subsubsection{The algebraic limit for iteration of partially pseudo-Anosov maps}  \label{sec:partialpseudoAlimit} 

 Now we turn to the case under consideration,  in which $\alpha \in \Mod S$ is  partially pseudo-Anosov  as described in Section~\ref{brockexamples} above. Thus $S $ is now a closed surface separated into two components $R$ and $L$ by  a simple closed curve $\sigma$
 and $\alpha \in \Mod (S) $  is such that $\alpha |_{L}$ is the identity and $\alpha  |_{R} = \chi $ is a pseudo-Anosov diffeomorphism of
$R$ preserving $\sigma$. 

Given $X \in \Teich(S)$  we set  $G_n =  G(\alpha^n(X), X)$ and $M_n =   Q(\alpha^n(X),X)$, so that  $M_n$  is the manifold  such that   $\Omega^+/G_n$ is conformally equivalent to $\alpha^n(X)$ while $\Omega^-/G_n$ is anti-conformally equivalent to $X$. In particular, $G_0$ is a Fuchsian  group uniformizing $X$ and our sequence of representations are the maps $\rho_n \co G_0 \to G_n$.

Brock showed, ~\cite{brock-itn} Theorem 5.4, that   the
representations $\rho_n$ 
converge algebraically to a representation $\rho_{\infty} \co G_0 \to G_{\infty}$ where $G_{\infty}$ is a geometrically infinite surface group with 
corresponding manifold $M_{\infty}$.
The regular  set $\Omega_{\infty}$ of $G_{\infty}$  has one $G_{\infty}$-invariant simply connected component $\Omega_{\infty}^-$ such that $\Omega_{\infty}^-/G_{\infty}$ is conformally equivalent to $X$. The positive end (corresponding to $\Omega^+$) has degenerated: if $g_{\sigma} \in \pi_1(S)$ corresponds to the separating curve $\sigma$, then  $\rho_{\infty}(g_{\sigma})$ is an accidental parabolic    and $\Omega_{\infty}^+$ collapses to a countable collection of simply connected components
$\Omega^{+,i}$ whose stabilisers are each conjugate to $\rho_{\infty}(\pi_1(L))$, so that $\Omega^{+,i}/ G_{\infty}$  is  a Riemann surface topologically equivalent to $\Int L$ for each $i$, with  $\rho_{\infty}(g_{\sigma})$ representing a loop encircling a puncture. 
Correspondingly, the upper end of $M_{\infty}$ is partially degenerate; the part corresponding to $L$ is geometrically finite   while the part corresponding to $R$ is degenerate with ending lamination the unstable lamination of $\chi$.

Let $H$ be a horocyclic neighborhood of the cusp 
corresponding to $\rho_{\infty}(g_{\sigma})$
in  $M_{\infty} $.
The assertion of~\cite{brock-itn} Theorem 5.4 is that the end of $M_{\infty} \setminus H$ cut off by the surface $R$ is asymptotically isomorphic to the end $E(\Sigma,\chi)$ described in the previous section, where 
 $\Sigma$ is a hyperbolic surface with the same topology as $\Int R$ but  equipped with a complete hyperbolic structure
so that the boundary   $\sigma = \partial R$  is replaced by a cusp on $\Sigma$ and $\chi = \alpha_{|R}$.

 \subsubsection{Models of the approximating manifolds}\label{sec:approxmodels}

Minsky \cite{minsky-elc1} \S 6.5 contains a description of the convex core of $M_n = Q(\alpha^n(X),X)$
in terms of a uniformly bi-Lipschitz model for its convex core $\CC_n$.  Fix $X \in \Teich (S) $ such that $\ell_{X}(\sigma) < \epsilon_0$ for some $\epsilon_0$ less than the Margulis constant. As above $\sigma$ separates $X$ into surfaces $R,L$; when needed we distinguish between the topological surfaces $R,L$ and the hyperbolic structures $X_R, X_L$ induced from $X$.

Also pick a complete hyperbolic surface  $\Sigma$ with the same topology as $\Int R$ but  
so that the boundary   $\sigma = \partial R$  is replaced by a cusp on $\Sigma$.  Let $\Sigma^c$ denote $\Sigma$ with a small (open) neighborhood of the cusp removed so that the boundary curve, which we denote $\sigma^c$,  has length $\epsilon_0$.

First we make a model $B_n$ for the part of $\CC_n$ corresponding to $ R$.
Let
$N= N_{\chi} $  be the hyperbolic 3-manifold with fiber $\Sigma$ and monodromy $\chi$. Let $N_n$ denote the cyclic $n$-fold cover of $N$, i.e.\,the manifold whose fundamental group is the kernel of the homomorphism $\pi_1(N) \rightarrow \pi_1(S^1) = \mathbb{Z} \rightarrow  {\mathbb{Z}}_n$.
Let $N_n^c$ denote $N_n$ with a small (open) neighborhood of the cusp removed so that the boundary curve of each fiber has  length $\epsilon_0$. Let $\hat N_n^c$ be 
 the manifold obtained by cutting $N_n^c$ open along a lift of some fiber and completing  metrically  to a manifold with boundary. Then 
just as described in Section~\ref{sec:pseudoAlimit}, there is a uniformly bi-Lipschitz homeomorphism from $\hat N_n^c$  to a model manifold $B_n= \Sigma^c \times [0,n]$,  in which the block $\Sigma^c \times [i,i+1]$ is isometric to $\Sigma^c\times [0,1]$ by a map homotopic to $\chi^{-i}$.  
In the model metric on $B_n$, the boundary loop $\sigma^c \times \{t\}$  has  length  $\epsilon_0$ for each $t \in [0,1]$ and the boundary cylinder $\sigma^c \times [0,n]$
has the obvious product of the Euclidean metrics on $\sigma^c$ and the interval $[0,n]$.
   (The discussion in  \cite{bowditch-ct} \S 8, especially Proposition 8.6 for the discussion of $\partial \Sigma$, explains the model  for $N_n^c$  in a neighborhood of a puncture.)

The model of the part of $\CC_n$ corresponding to  $ L$   is
essentially the product metric on   $C= X_L \times [0,1]$. We modify the metric on $X_L$ slightly to ensure the  boundary
 circles $\partial L \times \{t\}$ all  have length   $\epsilon_0$ and take the standard Euclidean metric of length one on the second factor.

Now we can make model $K_n$ for the whole of $\mathcal C_n$. Glue the circle $\partial L \times \{ 0 \} \subset C_n$  to the circle  $ 
\partial R \times \{ 0 \}= 
\sigma^c \times \{ 0 \} \subset B_n$, and likewise glue the  circle  $\partial L\times \{ 1 \}$ to the  circle 
 $\sigma^c \times \{ n \}$.  Let $K_n^c$ denote the resulting space. Let $\eta$ denote a circle of length $(n+1)$ obtained by moving in the  direction of the second factor  in $K_n^c$.   Finally, let  $K_n$ be  the manifold (with boundary)
obtained by (hyperbolic) Dehn filling $K_n^c$  with a Margulis tube $T_n$ with meridian $\eta$ and longitude $\sigma^c$, smoothing  out at the boundary if needed.  
Figure~\ref{fig:potbelly}  shows a `cross-section' of $K_n$. 
Note that  
the manifolds $K_n$ have uniformly bounded geometry away from $T_n$.

The lower  boundary of $K_n$,   denoted $\dd K_n^-$,  is obtained by gluing $\partial L \times \{ 0 \}$   to   $\partial R \times \{ 0 \}$. Thus we have an obvious map  $ \phi_n \co S \to \dd K_n^-$  which we use  to mark $K_n$.  If $s_0 \in \sigma$ is the base-point of $S$, we denote the image $\phi_n(s_0) = (s_0,0) \in \partial L \times \{ 0 \}$   by $ o_n$.
Lifting everything to universal covers, identifying $\til S$ with $\HHH^2$, and thinking of $\til K_n \subset \HHH^3$, we can arrange that $s_0$ lifts to $O \in \HHH^2$ and $o_n$ lifts to $O_n = O \in \HHH^3$. 

The upper  boundary of $K_n$,   denoted $\dd K_n^+$,  is obtained by gluing $\partial L \times \{ 1 \}$   to   $\partial R \times \{ n \}$. This gives a second obvious embedding 
$ \phi_n^+ \co S \to \dd K_n^+$ and we write $o_n^+ = \phi_n^+(s_0)$, with lift
$O_n^+ \in \HHH^3$.
Occasionally we write $\o_n^-$ for $o_n$, $O_n^-$ for $O_n= O$ and  $\phi_n^-$ for $\phi_n$ for clarity.

To see that the model manifolds $K_n$ are bi-Lipschitz equivalent to the convex cores $\CC_n$, 
note that the  marked surfaces $\dd K_n^{\pm}$ of $K_n$
are conformally  a uniformly bounded Teichm\"uller distance from the surfaces
$\alpha^n(X), X$ respectively, precisely as in the case of the manifolds $M_n = Q(\alpha^n(X), X)$.
Thus the standard techniques used in the proof of the ending lamination theorem  show that there  are bi-Lipschitz  homeomorphisms  between $ \CC_n$ and $K_n$, with constants 
 uniform in   $n$.   These are the models we will use.

\begin{figure}[hbt] 
 \centering
 \includegraphics[height=8cm]{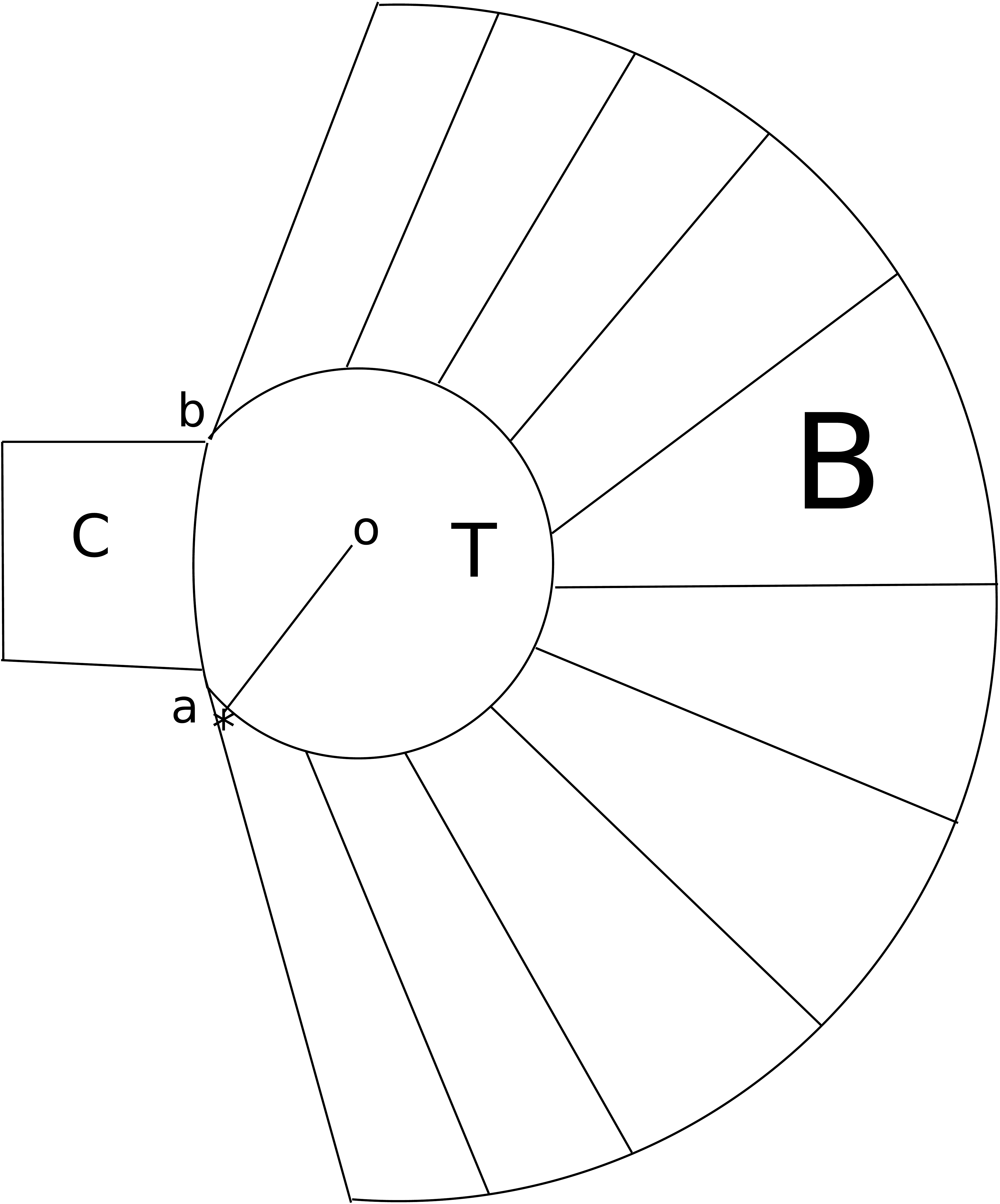}
\caption{A schematic picture of $K$ built up of $B$, $C$ and $T$. The points $a,b$ are the base-points $  \phi_n^-(s_0) =o_n^-$ and $ \phi_n^+(s_0) = o_n^{+}$ respectively. }
\label{fig:potbelly}
 \end{figure}

\subsubsection{The limit manifolds}\label{sec:limitmodels}
In the algebraic limit, the tube $T_n$ becomes a rank one cusp, see also~\cite{minsky-cdm}.  The lower boundary $\dd K_n^-$ of $K_n$ stays fixed but the upper boundary $\dd K_n^+$ develops into a partially degenerate end in which $L$ becomes a surface with a puncture. The part of the surface corresponding to $R$   becomes the degenerate end $E$ described in ~\ref{sec:approxmodels}.

In the geometric limit, the distance from $o_n^- \in \dd K_n^-$ to $o_n^+  \in \dd K_n^-$ stays bounded, because we can always travel through the bounded half $C = L \times [0,1]$. However to reach a point on a `middle' layer $\Sigma^c \times \{n/2\}$ of $B_n$  we have to travel ever further, either going directly `up' through $B_n$ or through $L \times [0,1]$, crossing $\partial L \times \{1\}$,  and then `down' from   $\Sigma^c \times \{n\}$ through $B_n$ to $\Sigma^c \times \{n/2\}$. Thus in the geometric limit, $K_n$ converges to  the manifold
$ S \times \mathbb R \setminus R \times \{0\}$  with two geometrically infinite ends, each asymptotically quasi-isometric to $E$. This is discussed in detail in~\cite{brock-itn}, but is not important for us here.

\subsection{Absence of Uniform Convergence}  \label{sec:nonUC}
Since the sequence $G_n$ does not converge strongly,  by Proposition~\ref{uep}, it must fail to satisfy UEP.  In fact it is easy to exhibit a sequence of points $g_k \in  G_0$ such that $|g_k| \to \infty$ while
$d(O, \rho_n(g_k) \cdot O) \leq c$ for all $n, k$ and some fixed $c>0$. (To see that this is equivalent to violating UEP, see~\cite{mahan-series1} Lemma 5.2.)

The meridian curve $\eta$ round
the boundary of the Margulis tube $T_n$ is split into  two homotopic paths $\tau_n, \upsilon_n$  in $K_n$ by the points $o_n^{\pm} \in \dd K_n^{\pm}$. 
The path  $\tau_n \co t \mapsto (s_0, t), t \in [0,n] $ joins $o_n^{\pm}$ going the `long' way round $\partial T_n$ in $B_n$, while the path $\upsilon_n \co t \mapsto (s_0, t), t \in [0,1] $ goes the `short' way round in $C_n$.

Let $s \mapsto \gamma(s) $ be a based loop homotopic to a fixed generator of $  \pi_1(S, s_0)$
and  lying entirely on $R$.   As in the proof of Lemma~\ref{endinglam}, the path $\gamma_n \co [0,1] \to K_n$, $\gamma_n (s)=   ( \chi^{n}\gamma(s), n) \in \Sigma^c \times \{n\} \subset \dd K_n^+$ has the same length $\ell$ say as the  path $s \mapsto \gamma_0(s) = (\gamma(s),0)\in \Sigma^c \times \{0\}  \subset \dd K_n^-$.

Now  the loops $\tau_n   \gamma_n  {\tau}^{-1}_n$ and  $\upsilon_n  \gamma_n   {\upsilon}^{-1}_n$ are  homotopic  in $K_n$, moreover from the above observation,  $\upsilon_n  \gamma_n  {\upsilon}^{-1}_n$ has length  $ 2 + \ell$ in $K_n$.  On the other hand    in $K_n$,  $\tau_n  \gamma_n  {\tau}^{-1}_n$ is homotopic to the loop $\chi^{n}(\gamma_0)  \subset \Sigma^c \times \{0\}\subset \dd K_n^-$, where by $\chi^{n}(\gamma_0)$ we mean  the path $ s \mapsto (\chi^n(s), 0) \in \Sigma^c \times \{0\}$.
By \eqref{lengths2}, the geodesic length of $\chi^{n}(\gamma_0)$ on $\Sigma^c \times \{0\}$ increases exponentially
with $n$; hence by the usual comparison of word length and geodesic length on
$\Sigma^c \times \{0\}$,  if  $g_n  \in G_0$ represents the loop  $ \chi^{n}(\gamma) \in \pi_1(S,s_0)$ then $|g_n| \to \infty$ in $G_0$.  Since $\rho_n$ is induced by the marking  $\phi_n \co s \mapsto (s,0) \in K_n^-$, the loop $ \chi^n(\gamma_0) $ is in the homotopy class of $\rho_n(g_n) \in \pi_1(K_n; o_n^-)$.

Let $O$ be the lift to $\HHH^3$ of $o_n^- \in K_n$ as above. Lifting the paths $\upsilon_n  \gamma_n  {\upsilon}^{-1}_n$, we have found a sequence $g_n \in G_0$  for which $d_{G_0}(1, g_n ) \to \infty$ but for which $d_{\HHH^3}(O, \rho_n(g_n) O ) $ is uniformly bounded.  As noted above, this violates UEP.

\subsection{Pointwise non-convergence}  \label{sec:nonPC}

Let $\Gamma = G_0$ be the Fuchsian group for which  $X = \HHH^2/\G$. 
Recall that $\sigma$ corresponds to a loxodromic $g_{\sigma} \in \G$  whose image under $\rho_{\infty}$ is parabolic.
Let $\mathcal P \subset \Lambda_{\G}$ denote the endpoints of  axes which project to $\sigma$, equivalently, the set of images under $\Gamma$ of the fixed points of $g_{\sigma}$.
The counter examples we seek for Theorem~\ref{brockexample} occur in the case of $\xi \in \Lambda_{\G}$ for which 
$\xi \notin \mathcal P $ but  $\hat i_{\infty}(\xi)= \hat i_{\infty}(p )$ for some $p \in \mathcal P$. Precisely which points these are is given by the following theorem of Bowditch.

\begin{theorem} [\cite{bowditch-ct} Theorem 0.2] \label{bowditch-ct} Let $ \HHH^2/\G$ be a punctured hyperbolic surface. Let $M= \HHH^2/G$ be a simply  
degenerate hyperbolic manifold
corresponding to a faithful type preserving representation $\rho \co \G \to G$, and suppose that there is a lower bound to the length of all loxodromics in $M$.
Suppose that $M$ has  ending lamination $\lambda$   
 and let $\hat i \co \Lambda_{\G} \to \Lambda_G$ be the corresponding $CT$-map. Then $\hat i (\xi) = \hat i (\eta), \xi,\eta \in  \Lambda_{\G} $ if and only if 
$\xi$ and $\eta$ are either either ideal 
end-points of the same  leaf of $\lambda$,  or ideal boundary points of a 
complementary ideal polygon of  $\lambda$.  \end{theorem}

This result was originally proved by Minsky~\cite{minsky-rigidity} in the (bounded geometry)
closed surface case. The condition on loxodromics means of course that the injectivity radius of $M$  is bounded below outside a horoball neighborhood of the punctures of $S$. 
This result has been extended to unbounded geometry and more general manifolds in \cite{mahan-split}, \cite{mahan-elct}, \cite{mahan-elct2}.

\subsubsection{The points of non-convergence} \label{sec:noncnvg}

The points of non-convergence of the maps $\hat i_n$ will be the endpoints of lifts of unstable leaves which bound the crown domain of the unstable lamination $|\lambda^u|$ of $\chi$ in $\Int R$. First, as mentioned above, we need to be careful about the precise meaning of saying that 
$\alpha_{|R} $ is pseudo-Anosov, so as  to ensure that  there is no  Dehn twisting around $\sigma$ when we consider $\chi$ as the restriction of $\alpha$.
We suppose given the hyperbolic surface $\Sigma$ as above and  a pseudo-Anosov map $\chi \co \Sigma^c \to \Sigma^c$ which pointwise fixes $\sigma^c$ and which is the identity in a horoball neighborhood of the cusp $\Sigma \setminus \Sigma^c$.
Then  $\chi$ induces an automorphism $\chi_*$ of $\pi_1(\Sigma^c,s_0) $, where we pick $s_0 \in \partial \Sigma^c$. 
Now identify  $R$ with $\Sigma^c$ and $\dd R = \sigma$ with $\partial \Sigma^c$.  With this identification, we insist that $(\alpha_*)_{|\pi_1(R,s_0)}= \chi_*$.

Continuing with the identification of $R$ with $\Sigma^c$,  note that the crown domain 
$F$ of the unstable lamination $\lambda^u$ of $\chi$  is an annulus  with one boundary component  $\sigma$ and the other  consisting of finitely many alternating segments of stable and unstable leaves. By taking a suitable power of $\chi$ if necessary, we can assume that these leaves map to themselves under $\chi$.

Next,  pick a lift $\tilde F$  of $F$ and a corresponding lift $\tilde \chi$ of $\chi$ which maps $\tilde F$ to itself and which is the identity on a particular lift 
$\tilde \sigma$ of $\sigma$. Let $\mu^u$ be the lift of  one of the unstable leaves bounding the lift $\tilde F$ and let $\xi^u$ be one of its endpoints in    $\Lambda_{\G}$.  From our assumption that $(\alpha_*)_{|\pi_1(R,s_0)}= \chi_*$,
it follows that $\chi(\xi^u) = \xi^u$. (Without this assumption, we might have $\chi(\xi^u) = \rho(g_{\s})^k(\xi^u)$ for some $k \in \mathbb Z$.)

The non-convergence part of Theorem~\ref{brockexample} is proved by
\begin{proposition}  \label{discont}  Let $\xi^u$ be an endpoint of a boundary leaf of $\tilde F$, and let $p \in \PP$ be an endpoint of the lift of $\sigma$  also bounding $\tilde F$. Then 
$ \hat i_{\infty}(\xi^u) =  \hat i_{\infty}(p)$, while no subsequence of the sequence 
$  \hat i_n(\xi^u)$ limits on $ \hat i_{\infty}(p)$.
\end{proposition}
\begin{proof} The statement that $ \hat i_{\infty}(\xi^u) =  \hat i_{\infty}(p)$  follows from Theorem \ref{bowditch-ct} since $\lambda^u$ is the ending lamination of $M_{\infty}$. The statement that  no subsequence of the sequence 
$  \hat i_n(\xi^u)$ limits on $ \hat i_{\infty}(p)$  is Corollary~\ref{unifqgeod1} which we prove below.
\end{proof}

To prove  Corollary~\ref{unifqgeod1}, we will construct, for each $n$, a quasi-geodesic    in the lift  $\til K_n$ of $K_n$ which passes through the basepoint $O   \in \til { K_n}$, and
with endpoints $ \hat i_n(\xi^u) $ and $ \hat i_{n}(p)$.
These quasi-geodesics will be uniform in $n$ and the result will follow.

We want to consider  $\xi^u \in \Lambda_{\G}$ as a point in the limit set $\Lambda^n_{\G}$ of the surface $
\alpha^n(X)$.  To do this, denote by $\HHH^2_n$ the universal cover of the surface $
\alpha^n(X)$,  with   basepoint $\tilde s_0 = \tilde \chi^n(\tilde s_0)  \in \tilde \sigma $. The map $\tilde \alpha^n\co \HHH^2 \to \HHH_n^2$ extends to  a homeomorphism
$h_n  \co \Lambda_{\G} \to \Lambda_{\G}^n$.
Since $\tilde \alpha^n(\mu^u) = \mu^u$, 
it follows that $h_n(\xi^u) = \xi^u$.

 In $\HHH^2_n$, let $ P_n$ be the foot of the  perpendicular from  $O = \tilde s_0$ to $\mu^u$ and consider the path $\beta_n$ which
follows the perpendicular from  $O$ to $P_n$  and then follows
$\mu^u$ from $P_n$ to its endpoint $\xi^u$. The segment  from  $O$ to $P_n$  has length  bounded independent of $n$ since outside the thin part of $X$, the diameter of the plaque $\tilde F$ is bounded. Hence $\beta_n$ is 
quasi-geodesic  in $\HHH^2_n$.

The marking of $K_n$ is given by the embedding  $\phi_n    \co  (S, s_0) \to  (\dd K_n^-, o_n^-)$ which lifts to $\til {\phi_n} \co (\HHH^2,O) \to (\widetilde K_n,O)$.  This extends to the map $ \hat i_n \co \Lambda_{\G} \to \Lambda_n$, where    $\Lambda_n$ is the limit set of $G_n$.
On the other hand, the
upper boundary  $\dd K_n^+ $ of $ K_n$  is marked by the map $\phi_n^+ = \phi_n  \circ \chi^{-n} $ whose lift $\til {\phi_n^+}  \co (\HHH^2_n,O) \to (\til {\dd K_n^+},O_n^+)$  extends to
 a map $q_n \co \Lambda^n_{\G} \to \Lambda_n$. Clearly, $q_n \circ  h_n = \hat i_n$, so in particular, $q_n(\xi^u)= \hat i_n (\xi^u)$.

\begin{lemma} \label{leafcontracts} The path $\til {\phi_n^+}(\beta_n)$ from $O_n^+$ to $\hat i_n (\xi^u) $
is quasi-geodesic  in $\widetilde  K_n$, with constants uniform in $n$.
\end{lemma}
\begin{proof} 

The segment of $\til {\phi_n^+}(\beta_n)$ from $O_n^+ $ to $\til {\phi_n^+}(P)$ has uniformly bounded length, so it is sufficient to show that $\til {\phi_n^+}(\mu_n)$ is uniformly quasi-geodesic  in $\widetilde K_n$.

Suppose first that we were dealing with the case of a pseudo-Anosov map $\chi$ on a punctured hyperbolic surface $Y$, so that the stable and unstable laminations $\lambda^s, \lambda^u$ of $\chi$  fill up $Y$.
Let $dx$ denote the transverse measure to $\lambda^s$ and 
$dy$ denote the transverse measure to $\lambda^u$, so that  $dx$  measures length along unstable leaves and $dy$  measures length along stable leaves. 
As in~\ref{sec:pseudo}, this defines a singular metric  on $Y$ which for brevity we  write as $ds^2 = dx^2 + dy^2$. Since $\chi$ expands along stable leaves, that is in the $y$-direction, the metric on 
$\chi^n(Y)$ is given by $ds^2 = c^{-2n}dx^2 + c^{2n}dy^2$.  The same formula defines a singular metric on the universal cover $\HHH^2$.

Restricting the model end $E =  Y \times [0,\infty)$  in~\ref{sec:pseudo} to $E_n=   Y \times [0,n]$  provides a model for the convex core of $Q(\chi^n(Y),Y)$.  We have obvious maps which embed $Y$ and $\chi^n(Y)$   in $E_n$ as pleated surfaces $Y_0 = Y \times \{0\},Y_n = Y \times \{n\}$ respectively. Passing to  universal covers,
as in \cite{minsky-rigidity}, see also \cite{CT},    the 
 metric in $\widetilde E_n$ is modelled by $ds^2 = dt^2 + c^{-2t} dx^2 + c^{2t} dy^2$, where $t$ is the `vertical' coordinate in the second factor.
 Thus the map  which projects $\widetilde E_n$ `vertically' upwards to $\widetilde Y_n  = \til Y \times \{n\}$ is a contraction when restricted to $\mu^u \times [0,n]$, where as above $\mu^u$ is a boundary leaf of $\lambda^u$.  Hence projecting from $\widetilde E_n$ to $(\mu^u,n)$ by first projecting `horizontally' in the surface $\widetilde Y_m$ to $(\mu^u,m)$ and then `vertically' to 
  $(\mu^u,n)$  is a contraction, from which the result (that a leaf of the unstable lamination on the top surface $\til Y_n$ is quasi-geodesic) follows by standard methods, see for example Bowditch \cite{bowditch-stacks} Lemma 4.2.  
 
  In the present case the model is somewhat more complicated because  the automorphism $\alpha$ of the underlying surface $S$ is partially pseudo-Anosov and $K_n$ limits on a partially degenerate end of $M_{\infty}$. 
  However we can apply the above argument working   in the space in which we electrocute  the left hand half $C_n = L \times [0,n]$, together with the Margulis tube $T_n$ around $\sigma_n$.

  An equivalent proof can be constructed  by modelling $\widetilde K_n$ as a tree of hyperbolic metric spaces as in ~\cite{mahan=jdg}. \end{proof}

\begin{center}
\begin{figure}
\includegraphics[height=10cm]{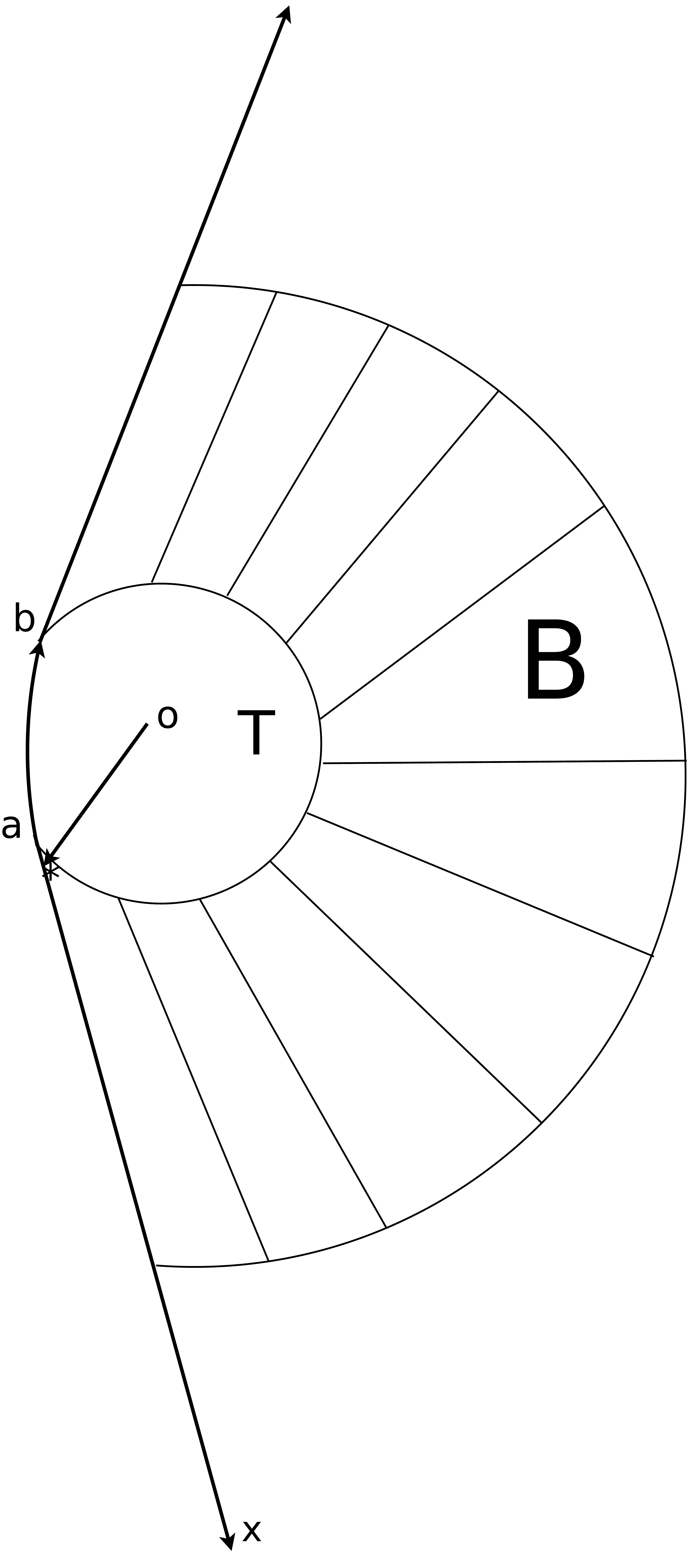} 
\caption{Geodesic realizations. The ray from $a =O_n^-$ to $x$ (the image of a geodesic ray in $\til  S$) lies on the lower boundary $ \partial \til K_n^-$ while its geodesic realization travels the short way round the (lifted) Margulis tube $\til  T$ to $b= O_n^+$ and thence runs along the upper boundary $\partial \til K_n^+$.  The endpoints of the two rays coincide in $\dd \HHH^3$.} 
\label{fig:geodrealzn} \end{figure}
\end{center}

Next, we modify the path $\til {\phi_n^+}(\beta_n)$ of the previous lemma to a quasi-geodesic  path from $O = O_n^-$ with the same endpoint $\hat i_n (\xi^u) \in \Lambda_n$, by prefixing it with  the path $\upsilon_n$  from $O_n^-$ to $O_n^+$ which goes the `short' way round $\dd T_n$ in $C_n$  as in section~\ref{sec:nonUC}. Since $\upsilon_n$ has uniformly bounded length $1$, the resulting path $\delta_n$ is a  $\til K_n$- quasi-geodesic from $O$ to $\hat i_n (\xi^u)$.
It follows that the geodesic ray from $O$ to $\hat i_n (\xi^u)$ either lies completely outside $T_n$, or enters $T_n$ only to exit at a point $O'$ a uniformly bounded distance from $O$.

\begin{lemma} \label{bangle} Let $p \in \mathcal P$ be as in the statement of Proposition~\ref{discont}.
 Let $\gamma_n$ be the hyperbolic ray from $O$ to $\hat i_n (p)$, and let $\delta_n$ be as above. Then the angle at $O$ between 
$\gamma_n$ and $\delta_n$ is uniformly bounded away from $0$.\end{lemma}

\begin{proof} First   consider first the limiting case in which $g_{\sigma}$ is parabolic.
After normalizing and 
working in the upper half space model $ \HHH^3$, we may assume we are in the following situation. Let $O \in \HHH^3$ be a fixed base point at Euclidean height $1$ above the base plane.
Suppose that $A \in \Isom \HHH^3$ is a parabolic  fixing $  \infty$. Suppose that $H$ is the height $1$ horoball at $\infty$, so that $O \in \dd H$.  
Suppose that  $\delta $ is a geodesic ray from $O$ which either lies completely outside $H$, or which enters $H$ and leaves it again at a point $O'$ at distance at most $k$ from $O$. Let $\gamma$ be the ray from $O$ to $\infty$.
Then the angle $\alpha$  between $\delta$ and  $\gamma$ at $O$ is bounded away from $0$; precisely  $2|\cot \alpha| \leq k'$, where $k'$ is the Euclidean bound on distance corresponding to the hyperbolic distance $k$.

Now we extend to the case of loxodromics of short translation length.
Working in the upper half space model $ \HHH^3$, let
$A_n \in \Isom \HHH^3$ be a loxodromic  fixing $  \infty$. Suppose that the translation length  $\ell(A_n)  \to 0$ as $n \to \infty$. Let $T_n$ be a constant distance cone around $\Ax A_n$, chosen so that the translation length of $A_n$ restricted to $\dd T_n$ is  a fixed length $\epsilon_0$.  Let $O \in \HHH^3$ be a fixed base point normalized to be at height $1$ and assume that $O \in \dd T_n$. Let $a_n$ be the other end point of $\Ax A_n$. Let $\gamma_n$ be the geodesic ray from $O$ to $a_n$ and let $\delta_n$ be another geodesic ray from $O$ which either lies completely outside $T_n$, or which enters $T_n$ and leaves it again at a point $O'$ at hyperbolic distance at most $k$ from $O$, where $k$ is bounded independent of $n$. 
We want to show that  the angle between $\delta_n$ and $\gamma_n$ at $O$ is bounded away from $0$.

Let $\theta_n$ be the angle between the sides of the cone $\til T_n$ and the horizontal.  Since $\ell(A_n)  \to 0$ while the translation length of $A_n$ restricted to $\dd \til  T_n$ is fixed,  $\theta_n \to 0$  as $n \to \infty$. 
Now the angle between $\gamma_n$ and the vertical at $O$ is $2 \theta_n$.
On the other hand, as is easy to compute, the angle between $\dd \til T_n$ and $\delta_n $ at $O$ is uniformly bounded away from $\pi/2$. Since $\dd \til  T_n$ is nearly horizontal, this proves the result.\end{proof}

\begin{cor} \label{unifqgeod1} No subsequence of the sequence 
$  \hat i_n(\xi^u)$ limits on the point $ \hat i_{\infty}(p)$.
\end{cor}   
\begin{proof}   By Lemma~\ref{bangle},   the visual angle subtended by
 $\hat i_n (p)$ and $\hat i_n (\xi^u)$ at $O$ is uniformly bounded below away from $0$.  Since $p$ is the fixed point of an element of $\G$, by algebraic convergence we have 
$\lim_{n \to \infty} \hat i_{n}(p) =   \hat i_{\infty}(p)$  and the result follows.
\end{proof}

\begin{rmk} {\rm
In hindsight, Proposition \ref{discont}  is perhaps not too unexpected as the paths $\tilde \phi_n^+(\beta_n)$
 live on the top sheet of the approximating manifolds $\til K_n$
and these converge to a ray  whose limit does not lie in the limit set of the original
surface subgroup $\pi_1(S)$.}
\end{rmk}

 \subsection{Pointwise Convergence}

 We shall now establish that the $CT$-maps $\hat i_n \co \Lambda_{\G} \to G_n$ of the Brock examples
converge pointwise for all points $\xi \in \Lambda_{\G}$ other than those described in Proposition~\ref{discont}. 
We do this by applying the conditions  $EP(\xi)$ (Embedding of Points)  and $EPP(\xi)$ (Embedding of Pairs of Points) for pointwise convergence from~\cite{mahan-series1}. These are essentially the criteria UEP and UEPP, relaxed so as to allow for dependence on the limit point $\xi$. 

 \subsubsection{Convergence criteria}

In~\cite{mahan-series1} we described $EP(\xi)$ and $EPP(\xi)$ in relation to a sequence of  elements $g_i \in \G$ chosen so that $g_i \cdot O$ is a quasi-geodesic in $\Gr \G$ and so that 
$g_i \cdot O \to \xi$ in the Euclidean metric on the ball model $\mathbb B \cup \dd \mathbb B$. It is easily seen  that is equivalent to replace this  with a criterion on the geodesic ray $[O,\xi)$ from $O$ to $\xi$ in the universal cover $\HHH^2$ of $X = \HHH^2/\G$, where $X \in \Teich(S)$.

\begin{definition}\label{critep}  Let $\G$ be a Fuchsian group such that $X = \HHH^2/\G$ is a closed hyperbolic surface  and let $\rho_n\co \Gamma  \to G_n$ be a sequence
 of  isomorphisms to Kleinian groups $G_n$.  Suppose given a sequence of  $(\Gamma, G_n)$-equivariant embeddings   $\tilde \phi_n \co (\HHH^2,O)  \to (\HHH^3,O)$   which induce basepoint preserving   embeddings  $  \phi_n \co X \to M_n$ with $(  \phi_n)_* = \rho_n$. Let $\xi \in \Lambda_{\G}$ and let $[O,\xi)$ be the geodesic ray in $\widetilde X = \HHH^2$ as above.  
 \begin{enumerate} \item The pair $((\rho_n), \xi)$ is said to satisfy
  $EP(\xi)$  if there
exist functions  $f_{\xi}  \co \mathbb N \to \mathbb N$ and $M_{\xi} \co \mathbb N \to \mathbb N$, with
 $f_{\xi}(N) \rightarrow\infty$ as $N\rightarrow\infty$, such that for all $x
\in [O,\xi)$ outside  $B(O, N)$ in $\HHH^2$, $\widetilde \phi_n(x)$ is outside $B(O, f_{\xi}(N))$ in $\HHH^3$, for all $ n \geq M_{\xi}(N)$. 
   
\item The pair $((\rho_n), \xi)$ satisfies $EPP(\xi)$ if there exists a function  $f'_{ \xi}(N )\co \mathbb N \to \mathbb N$ such that $f'_{\xi}(N)\rightarrow \infty$ as $N\rightarrow \infty$,
and such that for any  subsegment $[x,y] \subset [O,\xi ) $ lying outside $B(O;N)$ in $ \HHH^2$,  
the $\HHH^3$-geodesic $[\widetilde  \phi_n (x),\widetilde  \phi_n(x)]$ lies outside $B (O; f'_{\xi}(N))$  in   $ \HHH^3$ for all $n \geq M_{\xi}( N)$, where $M_{\xi}$ is as in (1). \end{enumerate}
\end{definition}

Note that in these definitions, we do not assume that $M_{\xi}(N) \to \infty$ with $N$, in fact in the best situation, $M_{\xi}(N) = 1$.   
We have:
\begin{theorem}[\cite{mahan-series1} Theorem 7.3] \label{ptwisecrit} Suppose  that   $\rho_n : \G \rightarrow \PSL$ is a sequence
of discrete faithful representations converging algebraically to $\rho_\infty \co \G \rightarrow \PSL $, and suppose the corresponding $CT$-maps $\hat i_n \co  \Lambda_{\G} \to \Lambda_{G_n}$ exist, $n = 1,2 \ldots,\infty$.
Let $\xi \in \Lambda_{\G}$.  Then
$\hat i_n (\xi) \to \hat i_{\infty}(\xi) $  as $n \to \infty$
 if  $((\rho_n);\xi)$ satisfies $EPP(\xi)$.
\end{theorem}

\subsubsection{Verifying pointwise convergence.}

As above, we take $\G = G_0$ and $\rho_n \co G_0 \to G_n$ to be the Brock examples as in Section~\ref{sec:partialpseudoAlimit}.  
Sometimes it will be important to distinguish between the surface $S$ and the hyperbolic structure $X \in \Teich(S)$. Fixing such a structure $X$, we may take the dividing curve $\sigma$ to be geodesic on $X$. The restrictions $X_R, X_L$  of $X$ to $R,L$ are hyperbolic surfaces with geodesic boundary $\sigma$. The universal cover $\til S$ of $S$ is identified with $\HHH^2$ using the lift of the structure $X$. Since $K_n$ is a quasi-isometric model for the convex core of $Q(\chi^n(X),X)$, we can  identify the universal cover $\til K_n$ of $K_n$  with a convex subset of $\HHH^3$.

The representations $\rho_n$ correspond to a sequence of embeddings $\til \phi_n \co (\HHH^2,O) \to ( \HHH^3,O)$ which descend to the maps
$\phi_n \co (S,s_0) \to (X \times \{0\}, (x_0, 0)) \subset K_n$.  
The map $\til \phi_n$ extends to the $CT$-map $\hat i_n \co \Lambda_{\G} \to  \Lambda_{n}$.  Hence  if $\xi \in \Lambda_{G_0}$,
the ray $[O,\xi) \subset \HHH^2$ maps under 
$\til  \phi$ to a path $\til  \phi([O,\xi)) \subset \HHH^3$ joining $\til \phi(O) = O$ to $\hat i_n(\xi) \in \Lambda_{n}$.
We denote the  $\HHH^3$-geodesic with these endpoints  by $[O, \hat i_n(\xi))$.
We will prove convergence $\hat i_n(\xi) \to \hat i_{\infty}(\xi)$ by checking that $(\rho_n, \xi)$ satisfies  the condition  $EPP(\xi)$.

 \subsubsection{Electric metrics}
 
For $\xi \in \Lambda_{G_0}$, there are two possibilities for the geodesic ray $[O,\xi) \subset \til S$: either it is eventually contained in a fixed  lift of  $R$, or not. In the first case, translating by an appropriate element of $\G = \pi_1(S)$ we may assume without loss of
generality that the entire ray  $[O,\xi)$ lies   in a fixed lift of  $R$.

Now consider the model manifold $K_n$ and let $D_n = B_n \cup T_n \subset K_n$.  Since Margulis tubes are convex and since the tube $T_n$ separates $B_n$ from $C_n$,   it follows that each lift $\widetilde{D_n}$ of $D_n$ is uniformly quasi-convex
in $\widetilde{K_n}$.  Hence $(\widetilde{K_n},\DD_n)$ satisfies the conditions of 
Lemma~\ref{ea-genl}, where $\DD_n$ is  the collection  of lifts $\til{D_n}$ of $D_n$. 
Let 
$d_e^n$ denote the resulting electric metric on $\widetilde{K_n}$ with the collection $\DD_n$  electrocuted.

Since the curve $\sigma$ along which we cut $X$ is geodesic, the lifts to $ \HHH^2$ of $X_R$ are convex, moreover they are clearly uniformly separated. Let $d_e^S$ denote the induced electric metric
on $\HHH^2$ with lifts of $X_R$ electrocuted.  Clearly the ray $[O,\xi) \subset \HHH^2$ has infinite length in $d_e^S$  if and only if 
it is not eventually contained in a fixed  lift of  $ R$.

\begin{lemma} \label{electricqisom}The map $\widetilde\phi_n \co (\widetilde{S}, d_e^S) \to (\widetilde{K_n}, d_e^n)$ is a 
 quasi-isometry 
 with constants which are uniform in $n$. \end{lemma} 
\begin{proof} The map $\widetilde\phi_n$ is the lift to $\til{S}$ of the map which sends $x \in L$ to $(x,0) \in X_L \times [0,1]$ and  $x \in R$ to $(x,0) \in X_R \times [0,n]$.
The lifts of the complement of $X_L$ are electrocuted in $\widetilde{S}$ and the lifts of 
 the complement  of $X_L \times [0,1]$ are  electrocuted in $\widetilde{K_n}$.
 This result follows  since $X_L \times [0,1]$ has vertical thickness $1$ in $K_n$.
\end{proof}

\begin{cor} The geodesic ray $[O,\hat i_n(\xi)) \subset \widetilde K_n$ has infinite length in the electric metric $d_e^n$  if and only if 
the   ray $[O,\xi) \subset \til S$ is not eventually contained in a fixed  lift of  $R$.  \end{cor}

In the light of this corollary, the property of the  ray $[O,\hat i_n(\xi)) $ having infinite $d_e^n$-length depends only on $\xi$. Thus we have two cases to consider 
depending on whether $[0, \xi )$ has finite or infinite length in the metric $d_e^S$.
In both cases, to prove convergence $\hat i_n(\xi) \to \hat i_{\infty}(\xi)$, we will verify $EPP(\xi)$.

\subsubsection{Case 1: The length of $[0, \xi )$ in the electric metric $d_e^S$ is infinite.} 

Let $\xi \in \Lambda_{\G}$. First we prove $EP(\xi)$.  
Since $[O, \xi )$ has infinite $d_e^S$-length, no tail of $[O, \xi )$ is contained  in a lift of $X_R$. Hence $[O, \xi )$ either crosses $X_L$ infinitely often, or has an infinite tail ending in a single lift of $X_L$. It follows that   there exists a proper function $f_{\xi}: \mathbb{N} \rightarrow  \mathbb{N}$ such that
if $x \in [0, \xi)$ is at distance at least $N$ from $O$ in $\widetilde{S}$, then   $d_{e}^S (O, x) \geq f_{\xi}(N)$.

By Lemma~\ref{electricqisom}, the map $\widetilde \phi_n$ is a uniform quasi-isometry with respect to 
the respective electric metrics. Hence 
$d_e^n(O,\widetilde \phi_n(x)) \geq c'f_{\xi}(N)$ for some  constant $c'>0$.  
Now any two lifts of $C_n = X_L \times [0,1]$ in $ \widetilde K_n$ are  separated by a constant $c>0$ which is also uniform in $n$. Hence  $d_{\HHH^3} (u,v) \geq c d_e^n(u,v)$ for any points $u,v \in \widetilde K_n$. Absorbing the constants $c,c'$ into the function $f_{\xi}$, we have shown that  
$d_{\HHH^3}(O,\widetilde \phi_n(x)) \geq f_{\xi}(N)$, which is just the statement $EP(\xi)$.

Now we prove $EPP(\xi)$.  
Consider a segment $[a,b] \subset [O, \xi)$  such that  $d(O,y) \geq N$ for all $y \in [a,b]$.
From the above, $d_{\HHH^3}(O, \widetilde \phi_n(y)) \geq  f_{\xi}(N)$ for all $y \in [a,b]$.
In other words,  $\widetilde \phi_n([a,b])$ lies outside $B(O, f_{\xi}(N))$ in $\widetilde K_n$.

Now replace  $[a,b] $ by the electro-ambient geodesic with the same endpoints. By Lemma~\ref{ea-genl} , this is a bounded distance from  $[a,b] $. It follows that $[a,b] $ is an electric quasi-geodesic in $(\widetilde{S}, d_{e}^S)$.  Hence by Lemma~\ref{electricqisom}, $\widetilde \phi_n([a,b])$ is an electric quasi-geodesic in $(\widetilde K_n, d_{e}^n)$ with constants which are uniform in $n$. 

By Lemma~\ref{ea-genl} again, the $d_{e}^n$-electro-ambient geodesic obtained by replacing 
 intersections of $\widetilde \phi_n([a,b])$  with the
lifts $\widetilde D_n$ by hyperbolic geodesics in  $\widetilde D_n$ with the same endpoints,  is a uniform hyperbolic quasi-geodesic  in $\widetilde K_n $ with
the hyperbolic metric.   Hence $\widetilde \phi_n([a,b])$  is a bounded distance away from  the $\HHH^3$-geodesic with the same endpoints.
This proves $EPP(\xi)$.

\subsubsection{Case 2: The length of $[0, \xi )$ in the electric metric $d_e^S$ is finite.}

  As before, let $\PP \subset \Lambda_{G_0}$ denote the set of endpoints of lifts
  of the geodesic $\sigma$
  and let $\PP_\infty \subset \Lambda_{\infty}$ be the image of these points under $\hat i_\infty$.
 We will prove
\begin{proposition} \label{ptwisecnvg} If the length of $[0, \xi )$ in the electric metric $d_e^S$ is finite and if  $  \hat i_{\infty}(\xi) \notin \PP_\infty$, then  $\lim_{N \to \infty} \hat i_{n}(\xi)
 =  \hat i_{\infty}(\xi)$.
\end{proposition}

This will complete the proof of Theorem~\ref{brockexample}.  This proposition is the only point at which we use Theorem~\ref{thm:strong=unif2}.

As above,  $\Sigma$ is
 a surface with the same topology as $\Int R$ but  equipped with a complete hyperbolic structure
so that the boundary   curve $\sigma$  is replaced by a puncture on $\Sigma$. 
We shall prove Proposition~\ref{ptwisecnvg} by comparison with the behaviour of $CT$-maps for the sequence of quasi-Fuchsian groups  $F_n $ 
uniformizing $(\chi^n(\Sigma), \Sigma  )$  with corresponding manifolds 
$  Q(\chi^n(\Sigma), \Sigma  )$, so that in particular, $F_0$ is Fuchsian and $\HHH^2/F_0 = \Sigma$. (Here $\chi$ is  the same pseudo-Anosov map  $\alpha_{|R}$ as above,  extended as the identity map in a neighbourhood of the puncture on $\Sigma$.)
 As in Section~\ref{sec:pseudoAlimit}, up to possibly passing to a subsequence, these groups limit on a simply degenerate group $F_{\infty}$ with corresponding manifold $M_{\chi}$ whose ending lamination is the unstable lamination of $\chi$. By~\cite{ctm-renorm} Theorem 3.12 the convergence is strong. 
Hence by Theorem~\ref{thm:strong=unif2} the $CT$-maps  $\hat k_n \co \Lambda_{F_0} \to \Lambda_{F_n}$ converge uniformly to  the $CT$-map $\hat k_\infty \co \Lambda_{F_0} \to \Lambda_{F_\infty}$. 
Thus the sequence of representations
$\bar \rho_n \co \pi_1(\Sigma) \to  F_n$ satisfies $UEPP$. As above, we can model the convex cores of the manifolds $  Q(\chi^n(\Sigma), \Sigma  )$ by the restriction $E_n = \Sigma \times [0,n]$ of the end $E$ of $M_{\chi}$. These model manifolds $E_n$ are  marked by the embedding  $\psi_n \co \Sigma \to \Sigma \times \{0\}$, which lifts to base points preserving embeddings  $\til \psi_n \co (\til \Sigma, O) \to \til {(\Sigma \times \{0\}} ,O)\subset \til E_n$.

To use the comparison between the representations $\bar \rho_n$ of $F_0= \pi_1(\Sigma)$ and
 $\rho_n$ of $G_0 = \pi_1(S)$,
we need to make definite the precise relationship  between the limit sets $\Lambda_{F_0}$ and  $\Lambda_{G_0}$.  
By definition the component $\Omega^-(G_0)$ of the regular set  of $G_0$ projects to 
the Riemann surface $X$. Let 
 $\til R_0$ be a fixed component of the lift of $R$ to  
 $\Omega^-(G_0)$, and let 
 $J \subset G_0$ be its stabiliser, with corresponding limit set $\Lambda_J \subset \Lambda_{G_0}$.  Then $\til R_0/J$  can be identified with $X_R$ so that  $J=\pi_1(R)$.
 Let $V$ be a  bi-Lipschitz homeomorphism  $X_R \to \Sigma^c$. Clearly $V$  induces an isomorphism
 $V_* \co J \to   F_0$ and  a map $\til V \co \til R_0 \to \til \Sigma^c$ which (see~\cite{mahan-series1} Theorem 4.1) extends to a corresponding $CT$-map $\hat i_V \co  \Lambda_J \to  \Lambda_{F_0}$.

 \begin{lemma} \label{EPtransfers} If $\hat i_{\infty}(\xi) \notin \PP_\infty$, then  $(\rho_n, \xi)$ satisfies $ EP(\xi)$.
\end{lemma}
\begin{proof} 
 
Translating by
 an appropriate element of $G_0$ we may assume without loss of
generality that the geodesic ray  $[0, \xi )$ is contained in $\til R_0$ so that
$\xi \in \Lambda_{J}$.  
Write $\bar \xi = \hat i_V(\xi)$.  Since as above $\bar \rho_n$ satisfies $UEPP$, then certainly $(\bar \rho_n, \bar \xi)$ satisfies $ EP(\bar \xi)$.   Hence there is a strictly increasing function $f \co \mathbb N \to \mathbb N$ such that  if $x \in [O,\bar \xi)$  and $d(O, x) > N$ then $d_{\HHH^3}(\til \psi_n(x) , O) >  f (N)$. 

Now given $N \in \mathbb N$ and $\bar x \in [O,\bar \xi)$, consider the $\HHH^3$-geodesic segment $\lambda = [O,\til \psi_n(\bar  x) ]$ from $O$ to $\til \psi_n(\bar  x)$, and let  $\mathcal H(\lambda)$ denote the collection of horoballs traversed by $\lambda$.
Let $P_N(\bar x,n)$ be the total length of the geodesic segments in  $ [O,\til \psi_n(\bar x) ] \cap \HHH^3 \setminus \mathcal H(\lambda) $
and 
$Q_N(\bar x,n) =    | \mathcal H(\lambda)|  $ be the number of horoballs traversed by  $\lambda$.  We claim there exist a strictly increasing  function $g   \co \mathbb N \to \mathbb N$ and   $M_N \in \mathbb N$ such that 
\begin{equation}\label{claim}
P_N(\bar x,n) + Q_N(\bar x,n) \geq g (N)  \ \ \mbox{for  all}  \ \bar x \in [O,\bar \xi), \bar x \notin  B_{\HHH^3}(O,N) \ \mbox{and} \ n \geq M_N.\end{equation}

If the claim is false, then there exists $K>0$ such that for all $N$, there exist $ \bar x_N \in [O,\bar \xi)$, $\bar x_N \notin B_{\HHH^3}(O,N)$ and arbitrarily large $n  \in \mathbb N$ such that  
\begin{equation}\label{claimfalse} P_N(\bar x_N,n_N) + Q_N(\bar x_N,n_N) \leq K. \end{equation} Inductively, choose $n = n_{N } > n_{N-1}$. Then \eqref{claimfalse} implies 
in particular that there is a uniform bound to the number of horoballs traversed by 
the ray $[O,\til \psi_{n_N}(\bar x_N) ]$. By slightly adjusting constants, we can assume that each horoball is penetrated to a distance at least $a>0$ for some $a$.  Now by hypothesis
$d(O, \til \psi_{n_N}(\bar x_N)) \geq  f (N)$. Thus there can be no uniform
upper bound to the distance travelled through each horoball, in other words, we can find a sequence $H_N$ of horoballs 
 such that the length of the segment $\til \psi_{n_N}(\bar x_N) \cap H_N$ tends to infinity with $N$.
By choosing the first such horoball traversed, we can assume that $H_N$ intersects $B_{\HHH}(O,K)$. Passing to a subsequence if necessary, we may assume that all horoballs $H_N $  are based at the point $\hat k_n(\eta)$ for some fixed parabolic point $\eta \in \Lambda_{F_0}$.
Thus we can find a sequence $y_N \in [O,\bar x_N]$ such that $\til \psi_{n_N}(y_N) \in H_N$ and 
$d(O, \til \psi_{n_N}(y_N)) \to \infty$. Since the rays  $\til \psi_{n_N}([O,\bar \xi))$ converge to 
$\til \psi_{\infty}([O,\bar \xi))$ uniformly on compact subsets in $\HHH^3$  (by $UEPP$ for the sequence $\bar \rho_n$), this means that 
$\til \psi_{n_N}(y_N)) \to \hat j_\infty(\eta)$. On the other hand,  $\til \psi_{n_N}(y_N) $ is arbitrarily close in the Euclidean metric   on $\mathbb B  \cup \Chat$ to $ \hat k_{n_N}(\eta)$ for large $N$. Hence $\hat k_{\infty} (\bar \xi)  = \hat k_{\infty} (\eta) $. 

Since $\eta$ is a parabolic point in $\Lambda_{F_0}$, by Bowditch's Theorem~\ref{bowditch-ct} this means that either $\bar \xi \in V_*(\mathcal P)$ or $\bar \xi $ is the endpoint of a leaf in the crown of the unstable lamination of $\chi$. Since $\hat i_V \co \Lambda_J \to \Lambda_{F_0}$ is one-to-one  except on $\mathcal P$, the same is true of $\xi$. 
Since by assumption $\xi \notin \mathcal P$, we deduce that  $\xi$ is the end of a boundary leaf of the crown of $\chi$, which gives, 
using Theorem~\ref{bowditch-ct} again, $\hat i_{\infty} ( \xi) \in \PP_{\infty}$, contrary to hypothesis. This proves claim \eqref{claim}.

Now we will show that  claim \eqref{claim} implies that $(\rho_n,\xi)$ satisfies $EP(\xi)$.
As above, let $D_n = B_n \cup T_n$ and  let $\til D_n$ denote  the lift of  $D_n$ corresponding to $\til R_0$ above, that is, whose stabiliser is $\rho_n ({J})$. Let $\til B_n$ be the corresponding lift of $B_n$.
 The map $V$ induces an obvious uniformly bi-Lipschitz  map $V_n \co B_n = X_R \times [0,n] \to E_n^c = \Sigma^c \times [0,n]$,  where  $E_n  = \Sigma \times [0,n]$ is the model of the convex core of   $Q(\chi^n(\Sigma), \Sigma)  $ as in Section~\ref{sec:approxmodels}.
 Clearly $V_n \circ \phi_n = \psi_n \circ V$, while on the level of fundamental groups, $(V_n)_*\circ  \rho_n = \bar \rho_n \circ  V_*$  and $(V_n)_*,   V_*$ are group isomorphisms.

Since Margulis tubes are convex, it follows  as in~\cite{farb-relhyp} that  $\til D_n$
satisfies the condition of Lemma~\ref{ea-genl} relative to the collection $\TT_n$ of lifts of $T_n$ it contains, as does $\til E_n$ relative to the set of  horoballs $\mathcal H_n$ say.
Let $\til D_n^{e}, \til E_n^e$ denote the corresponding electric spaces.
Note that $V_*$ induces a bijective correspondence between $  \mathcal T_n$ and $\mathcal H_n$.

To avoid  having to define the extension of $\til V_n$ to the whole of 
$\til D_n^e$ we  proceed as follows.  Suppose that $\lambda$ is an electric quasi-geodesic in $\til D_n^e$ with endpoints in $\til B_n$.    
Replace $\lambda$ with a path $\hat \lambda$ which runs along the boundaries of the electrocuted sets in $\mathcal T_n$ as follows.  Suppose some  segment $\lambda'$ of $\lambda$ enters and leaves some $T \in \mathcal T_n$ at points $a,b$ respectively.  Replace $\lambda'$ by the segment $(a, [0,1]) \cup (b, [0,1])  \subset \partial T \times [0,1]$ of electric length $2$. 
Since the sets in $\mathcal T$ are uniformly separated, the resulting path $\hat \lambda$ is still an electric quasi-geodesic. Now extend the definition of $\til V_n$ to  a map, still denoted $\til V_n$, which sends $\partial T \times [0,1] \to \partial H \times [0,1] $ in the obvious way, where $H \in \mathcal H$ corresponds to $T \in \mathcal T$.
Using the fact that $V_n$ is uniformly bi-Lipschitz, it is easy to see that $\til V_n(\hat \lambda)$ is an electric quasi-geodesic in $\til E_n^c $, and that the number of electrocuted components traversed by $\hat \lambda$ and $\til V_n(\hat \lambda)$ is the same.

Now suppose $\xi \in \Lambda_J$ as in the statement of the Lemma.  Since $\xi \notin \mathcal P$ the ray $[O,\xi)$ is contained in the convex hull of $\til X_R \subset \til X$.
Note however that the path $\til V([O,\xi))$ may not be a quasi-geodesic  in $\til \Sigma$ as it is contained in $\Sigma^c$ and thus may skirt round the boundaries of horoballs in $\Sigma$.  We can nevertheless  work with the ray $\til V([O,\xi))$,  which ends at the point $\hat i_V(\xi) = \bar \xi$, see for example~\cite{mahan-series1} Theorem 4.1.
(We obtain a quasi-geodesic ray from  $\til V([O,\xi))$ by replacing each segment which skirts a horoball with the corresponding geodesic joining the entry and exit points, see for example~\cite{mahan-series1}  especially Lemma A.5.)

 Let $x \in [O,\xi)$ and let $\bar x= \til V(x)$.
 We want to compare the ray $[O, \til \phi_n(x) ] \subset \til K_n$ to the ray $[O, \til \psi_n(\bar x) ] \subset \til E_n $.
  Since $\til D_n$ is quasi-convex in $\til K_n$, we can after bounded adjustments assume that $ [O, \til \phi_n(x) ]  \subset \til D_n$.
Since  $\til V_n  \til \phi_n = \til \psi_n \til V $ we have $\til V_n(\til \phi_n(x) ) =  \til \psi_n (\bar x)$.
  
Replacing the electric geodesic $\lambda$ say from $O$ to $ \til \phi_n(x)$ in $ \til D_n^e$ by the corresponding electric quasi-geodesic $\hat \lambda$ as above, we see that 
$\til V_n (\hat \lambda)$ is a well-defined electric quasi-geodesic  in $ \til E_n^e$ with endpoint $\til \psi_n (\bar x)$. Moreover $\til V_n (\hat \lambda)$ has length comparable to $\hat \lambda$. Since  \eqref{claim} effectively says that the length of $\til V_n (\hat \lambda)$ in the electric metric on $ \til E_n^e$ goes to infinity uniformly with $N$  independently of $n$, the same is true of $\hat \lambda$.
This proves that $(\rho_n, \xi)$ satisfies $EP(\xi)$ and we are done.
\end{proof}

\begin{cor} If $\hat i_{\infty}(\xi) \notin  \PP_\infty$, then  $(\rho_n, \xi)$ satisfies $ EPP(\xi)$.
\end{cor}
\begin{proof} Continuing with the notation of Lemma~\ref{EPtransfers}, 
 let $  \lambda$ be a geodesic segment in $[O,   \xi)$ outside $B_{\HHH^2}(O,N)$
and let $\bar \lambda = \til V(\lambda)$.  
As in the previous lemma, note that $\bar \lambda$ may not be a geodesic as it is contained in $\Sigma^c$ and thus may skirt round the boundary of a horoball in $\Sigma$.  This leads to  an annoying technical issue in that it is convenient only to work with segments $\bar \lambda$ whose endpoints  lie outside the lifts of the horoball  $\Sigma \setminus \Sigma^c$. 
To fix this, note that  in Definition~\ref{critep} of condition $ EPP(  \xi)$ for convergence, it is clearly enough to check the condition  for an increasing sequence of values  $N_1< N_2 < \ldots$.  Since $\til V([O,   \xi))$ does not terminate in the cusp, we may therefore restrict to those $N_i$ for which the first $x \in \til V([O,   \xi))$ outside $B(O,N_i)$ is outside a horoball. Thus given $\bar \lambda$ as above, by extending  forwards and backwards along $\til V([O,   \xi))$  if necessary, we may  assume that its initial and final points are outside horoballs in $\til \Sigma$.

Now consider the geodesic $[\til \psi_n ( \bar\lambda)]$ in $\til E_n^e$ and let  $\mu, \mu_{ea}$ be respectively the electric geodesic and the  electro-ambient geodesic with the same endpoints, as in Section~\ref{sec:relhyp}.
By Lemma~\ref{ea-genl},  $\mu_{ea}$   is 
a bounded distance from  $[\til \psi_n (\bar \lambda)]$.  Using $ EPP( \bar\xi)$, we deduce that $  \mu_{ea}$  is outside 
$B_{\HHH^3}(O,g(N)-k)$, for some uniform $k>0$. Then using the same method as in the previous lemma, it follows that any point on $\mu$ is outside some ball $B(O,h(N))$ in the electric metric on $\til E_n^e$ for some function $h(N) \to \infty$ with $N$.  

 Now using the same trick as in the previous lemma,
replace $\mu$ with the electric quasi-geodesic $\hat \mu $ and apply the map $\til V_n^{-1}$. We obtain an electric quasi-geodesic  $\hat \nu = \til V_n^{-1}(\hat \mu)$  in $\til D_n^e$ with the same endpoints as
$\til \phi_n (  \lambda)$. Since $\til V_n^{-1}$ is bi-Lipschitz with respect to electric metrics, for any point $Q \in \hat \nu$ we have $d_e(O,Q) \succ h(N)$, where we write $X \succ Y$ to mean there is a uniform constant $c>0$ such that $X > cY$.

By Lemma~\ref{ea-genl} again, it will be enough to show that  the electro-ambient quasi-geodesic  obtained from $\hat \nu$  by replacing each segment which cuts through an equidistant tube $T \in \mathcal T_n$ with the hyperbolic geodesic with the same endpoints, is  outside some ball $B_{\HHH}(O,f(N))$ for some function $f(N) \to \infty$ with $N$.  

Suppose that $A$ and $B$ are the entry and exit points of $\hat \nu$ to some $T \in \mathcal T_n$. 
Suppose that the  hyperbolic geodesic from to 
 $O$ to $A$ first meets $T$ at a point $\bar A$. 
 Since  $d_e(O,A) \succ h(N)$, it follows that 
 $d_{\HHH}(O,\bar A) \succ h(N)$. 
We deduce  from Lemma~\ref{penetrates} below  that 
$T$ is entirely outside $B(O,R)$ for some $R \succ h(N) $. In particular
 the hyperbolic geodesic segment $[A,B]$ is outside
 $B(O, R)$ and the result follows.
  \end{proof}

\begin{lemma} \label{penetrates} Suppose that $T$ is an equidistant tube in $\HHH^3$, that is, the set of points equidistant from a geodesic  axis in $\HHH^3$, and that $T$ has radius at least  $r$ for some uniformly large $r$.  Suppose that $A \in \dd T$
 is outside $B(O,R)$, where $O \in \HHH^3$ is  a  fixed base-point. Then the entire tube $T$ is outside $B(O,R')$ for some $R' \succ R$.
\end{lemma}
\begin{proof}  Let $P$ be the point on $T$ nearest to $O$ in the hyperbolic metric   and let $P',A'$ be the feet of the perpendiculars from $P, A$ to the axis of $T$. We claim that   if $d(A,P) > c >1$, then the angle $\theta$ between the geodesics $[A,P]$ and $[A,A']$ is  uniformly bounded away from $\pi/2$.   If $d(A',P') > 1$, this is easy since $[A,P]$ roughly tracks the quasi-geodesic $[A,A']  \cup [A',P'] \cup [P',P]$. 

If $d(A',P' ) \leq 1$ let $Q,Q'$ be respectively the feet of the perpendiculars from $A', P'$ to the geodesic $[A,P]$, so that $|QQ'| <1$.   Since
 $\cos \theta = \tanh |AQ| / \tanh |AA'|$ and $|AA'|\geq r$, it follows that $\cos \theta$ is bounded away from $0$ unless $|AQ|$ is very small. 
By symmetry $|A P| = 2|A Q| +|QQ'| $, so if we assume that 
$|A P| > c >1$  this is impossible. 
Since  $[O,A]$ is outside $T$, and since $[A,A']$ is perpendicular to $\dd T$ at $A$, we have shown that either $[A,P]$ has uniformly bounded length, or the angle between 
$[O,A]$ and $[A,P]$ is bounded away from $0$.

It follows in all cases that $[O,A] \cup [A,P]$ is  a uniform quasi-geodesic  and hence that 
$d(O,P) \succ R$. The result follows by convexity.
\end{proof}

\begin{cor} If $\hat i_{\infty}(\xi) \notin \PP_\infty$, then  $\hat i_n(\xi) \to \hat i_{\infty}(\xi)$.
\end{cor}
\begin{proof} This follows immediately from Theorem~\ref{ptwisecrit}.
\end{proof}

This completes the proof of Proposition~\ref{ptwisecnvg}.

\begin{rmk} {\rm In   the proof of Theorem~\ref{brockexample},  we used the pseudo-Anosov $\chi = \alpha |_{R}$ only to get a simply degenerate manifold
corresponding to a representation of $\pi_1(R)$ in the algebraic limit. We could replace the sequence $G_n$ with any sequence of representations $\rho'_n$ of $\pi_1(S)$
such that \\
a) the sequence $\rho'_n|_{\pi_1(R)}$ converges to a simply degenerate representation of $\pi_1(R)$ and \\
b) the sequence $\rho'_n|_{\pi_1(L)}$ converges to a quasi-Fuchsian representation of $\pi_1(L)$. \\
Then the general form  \cite{mahan-elct2} Theorem \ref{bowditch-ct}, which applies to the case in which the geometry of the limit manifolds do not necessarily have bounded geometry, together with a suitably modified version of Theorem~\ref{thm:strong=unif}, would furnish   the same conclusion as Theorem \ref{brockexample}, where we replace the unstable lamination of $\chi$ with the ending lamination of $G_{\infty}$. }
\end{rmk}


\begin{thebibliography}{000}


\bibitem{agol} I. Agol.  
\newblock Tameness of hyperbolic 3-manifolds. 
\newblock {\em arXiv:math.GT/0405568}, 2004.

\bibitem{and-can} J.  Anderson and D. Canary.
\newblock  Algebraic limits of Kleinian groups which rearrange the pages of a book.
 \newblock {\em Invent. Math.},  126,  205Ð214, 1996.
 
 
     \bibitem{and-can1} J.  Anderson and D. Canary.
  \newblock Cores of hyperbolic 3-manifolds and limits of Kleinian groups.
\newblock {\em Amer. J. Math.}, 118, 745-779, 1996.


\bibitem{and-can-mc} J.  Anderson,  D. Canary and D. McCullough.
  \newblock On the topology of deformation spaces  of Kleinian groups.
\newblock {\em Ann. of Math.}, 152, 693-741, 2000.




 
 
 \bibitem{bon}
F.~Bonahon.
\newblock Bouts des vari\'et\'es de dimension $3$.
\newblock {\em Ann. of Math.}, 124,   71--158, 1986.



\bibitem{bowditch-relhyp} B. H. Bowditch. 
\newblock Relatively hyperbolic groups. 
\newblock {\em   Internat. J. Algebra and Computation}, 22,  1250016-1 -- 1250016-66, 2012.
 
 

\bibitem{bowditch-ct} B. H. Bowditch. 
\newblock The {C}annon-{T}hurston map for punctured surface groups
\newblock {\em Math. Z.}, 255, 35 -76, 2007. 

\bibitem{bowditch-stacks} B. H. Bowditch.
\newblock Stacks of hyperbolic spaces and ends of $3$-manifolds.
\newblock   
In C. Hodgson \emph{et al} eds, {\em Geometry and Topology Down Under}, Amer. Math. Soc. Contemporary Mathematics Series, to appear. 

 
  \bibitem{brock-itn} J. Brock.
  \newblock Iteration of mapping classes and limits of hyperbolic 3-manifolds.
\newblock {\em Invent. Math.}, 1043, 523 -- 570, 2001.


  \bibitem{minsky-elc2} J. Brock, R. Canary, Y. Minsky.
    \newblock The classification of Kleinian surface groups, II: The Ending Lamination Conjecture.
    \newblock   {\em Ann. of Math.}, 176, 1-142, 2012.
    
    
    \bibitem{calegari} D.  Calegari and D. Gabai.
\newblock  Shrinkwrapping and the taming of hyperbolic 3-manifolds.
 \newblock {\em J. Amer. Math. Soc.},  19,  38 -- 446, 2006.


\bibitem{CEG}
   R.~Canary, D.~Epstein and P.~Green.
\newblock Notes on notes of {T}hurston.
\newblock In D.~Epstein, ed., {\em Analytical and
Geometric Aspects of
    Hyperbolic Space}, LMS Lecture Notes 111, 3--92. Cambridge
University
    Press, 1987.
    
    
      \bibitem{CM} R.~Canary, Y. Minsky
  \newblock On limits of tame hyperbolic 3-manifolds.
\newblock {\em J. Differential Geom.},  43, 1-41, 1996.



  \bibitem{CT} J. Cannon and W. P. Thurston.
       \newblock Group {I}nvariant {P}eano {C}urves.
       \newblock {\em Geom. Topol.},  11, 1315-1355, 2007.
       
       \bibitem{mahan-elct2} S.~Das and M.~Mj.
\newblock {Addendum to Ending Laminations and Cannon-Thurston Maps:
  Parabolics}.
\newblock {\em arXiv:1002.2090}, 2010.
       


     
    
    \bibitem{EpM} D. B. A. Epstein and A. Marden.
\newblock  Convex hulls in hyperbolic
  space, a theorem of Sullivan, and measured pleated surfaces. 
 \newblock In D.~Epstein, ed., {\em Analytical and
Geometric Aspects of
    Hyperbolic Space}, LMS Lecture Notes 111, 3--92. Cambridge
University
    Press, 1987. 


    
    
    
\bibitem{evans1} R. Evans.
\newblock Deformation spaces of hyperbolic 3-manifolds: strong convergence and
tameness.
\newblock {\em Ph.D. Thesis, Unversity of Michigan}, 2000.

 
 \bibitem{evans2} R. Evans.
 \newblock Weakly Type-Preserving Sequences and Strong Convergence.
\newblock {\em Geometriae Dedicata}, 108, 71--92, 2004.



\bibitem{FLP}
A.~Fathi, P.~Laudenbach, and V.~Po{\'e}naru, 
\newblock {\it Travaux de {T}hurston sur les surfaces}, 
  Ast{\'e}risque 66--67, 
 Soci{\'e}t{\'e} Math{\'e}matique de France (1979).

\bibitem{farb-relhyp} B. Farb.
\newblock Relatively Hyperbolic groups.
\newblock {\em Geom. Funct. Anal. }, 8, 810-840, 1998.


\bibitem{floyd} W. Floyd.
\newblock Group completions and limit sets of Kleinian groups.
\newblock {\em Invent. Math.}, 57, 205-218, 1980.




\bibitem{jor-mar} T. J{\o}rgensen and A. Marden. 
\newblock Algebraic and Geometric convergence of
                {K}leinian groups.
                \newblock {\em Math. Scand.}, 66, 47--72, 1990.
                
                
 \bibitem{kap} M.~Kapovich.
\newblock {\em Hyperbolic  manifolds and discrete groups.}
\newblock Birkh\"auser, 2000.
    



\bibitem{marden-book} A.~Marden.
\newblock {\em Outer Circles: An introduction to hyperbolic $3$-manifolds.}
\newblock Cambridge University Press, 2007.


        
 \bibitem{mms}
M.~Mccullough and A. ~Miller and G.~A.~Swarup.
\newblock {Uniqueness of cores of noncompact 3-manifolds}.
\newblock {\em J. London Math. Soc.} 32, 548-556, 1985.  



  \bibitem{ctm-renorm}
  C.~McMullen.
  \newblock Renormalization and $3$-manifolds which fiber over the
circle.
  \newblock  Annals of Math. Studies 142. Princeton University Press,
1996.




\bibitem{minsky-ends} Y. N. Minsky.
\newblock Teichm\"uller geodesics and ends of hyperbolic $3$-manifolds.
\newblock {\em Topology}, 32, 1 -- 25,
 1992.

\bibitem{minsky-rigidity} Y. N. Minsky.
\newblock On rigidity, limit sets, and end invariants of hyperbolic $3$-manifolds.
\newblock {\em J. Amer. Math. Soc.}, 7, 539 -- 588,
 1994.


\bibitem{minsky-elc1} Y. N. Minsky.
\newblock The {C}lassification of {K}leinian surface groups {I}: Models and Bounds.
\newblock  {\em Ann. of Math. }, 171,  1 -- 107, 2010.
 
 


 \bibitem{minsky-cdm} Y. N. Minsky.
\newblock End Invariants and the {C}lassification of Hyperbolic $3$-manifolds.
\newblock {\em  Current Developments in Mathematics},  Vol. 2002, 111 -- 141, 
 2002.



\bibitem{mahan=jdg} M. Mitra.
\newblock Cannon-Thurston Maps for Trees of Hyperbolic Metric Spaces.
\newblock {\em J. Differential Geom.},  48, 135 -- 164, 1998.
 


 
 \bibitem{miyachi} H. Miyachi.
\newblock Moduli of continuity of Cannon-Thurston maps.
\newblock  In Y.~Minsky, M.~Sakuma and
C.~Series eds., {\em Spaces of Kleinian groups}, LMS Lecture Notes 329, 121 --150. Cambridge University Press, 2006.   



\newblock Notes on notes of {T}hurston.
\newblock In D.~Epstein, ed., {\em Analytical and
Geometric Aspects of
    Hyperbolic Space}, LMS Lecture Notes 111, 3--92. Cambridge
University
    Press, 1987.


\bibitem{mahan-elct}
M.~Mj.
\newblock {Ending Laminations and Cannon-Thurston Maps}.
\newblock {\em preprint, arXiv:math.GT/0702162}, 2007.


\bibitem{mahan-pared} M. Mj.
\newblock Cannon-Thurston Maps for Pared Manifolds of Bounded Geometry.
\newblock {\em Geom. Topol.} 13, 189-245, 2009.


\bibitem{mahan-ibdd} M. Mj.
\newblock Cannon-Thurston Maps, i-bounded Geometry and a Theorem of McMullen.
\newblock {\em  Actes du  S\'eminaire de Th\'eorie spectrale et g\'eom\'etrie (Grenoble)}, 28, 63 -- 107,  2009-2010.



\bibitem{mahan=ramanujan} M. Mj.
\newblock Cannon-Thurston Maps and Bounded Geometry.
\newblock {\em  Ramanujan Math. Soc. Lect. Notes Ser.} 10, 489 -- 511,  2010.



\bibitem{mahan-kl} M. Mj.
\newblock Cannon-Thurston Maps for Kleinian Groups.
\newblock {\em preprint, arXiv:1002.0996}, 2010.

          
   \bibitem{mahan-series1}
M.~Mj and C.~Series.
\newblock {Limits of Limit Sets I}.
\newblock {\em Geometriae Dedicata}, 1--33, 2012.  
 

\bibitem{mahan-split} M. Mj.
\newblock Cannon-Thurston Maps for  Surface Groups.
\newblock {\em Ann. of Math.}, to appear, 2013.

   \bibitem{mahan-series3}
M.~Mj and C.~Series.
\newblock {Limits of Limit Sets III: The general case}.
\newblock {\em In preparation}.    

 

 



\bibitem{Otal}
J-P.~Otal,   {\it Le th\'eor{\`e}me d'hyperbolisation pour 
les vari{\'e}t{\'e}s fibr{\'e}es de dimension $3$}, 
  Ast{\'e}risque 235,
 Soci{\'e}t{\'e} Math{\'e}matique de France, (1996).

 
\bibitem{thurstonnotes} W. P. Thurston. \newblock The {G}eometry and {T}opology of 3-{M}anifolds.
\newblock {\em Princeton University Notes},
     1980.


\bibitem{thurston} W. P. Thurston.  
 \newblock 	Three dimensional manifolds, Kleinian groups and hyperbolic geometry. 
\newblock {\em Bull. Amer.  Math. Soc.} 50, 357 -- 382, 1982.


 
\end{thebibliography}

\end{document}